\documentclass[10pt,english]{article}
\usepackage[ansinew]{inputenc}
\usepackage[english]{babel}
\usepackage[T1]{fontenc}
\usepackage{lmodern}
\usepackage{amsmath}
\usepackage{array}
\usepackage{amssymb}
\usepackage{graphicx}
\usepackage{latexsym,bm}
\usepackage{amsthm}
\usepackage{amsfonts}
\usepackage{longtable}
\usepackage{multicol}
\usepackage{fancyhdr}
\usepackage[thinlines]{easytable}
\usepackage{caption}
\usepackage{MnSymbol}
\usepackage{todonotes}
\usepackage{multirow}
\usepackage{hhline}
\usepackage{rotating}
\usepackage{setspace}
\usepackage[colorlinks=true,linkcolor=blue,citecolor=blue]{hyperref}
\usepackage{subfig}
\usepackage[capitalise]{cleveref}
\usepackage{rotating}
\usepackage{tabularx}
\usepackage{dcolumn}
\newcolumntype{d}[1]{D{.}{.}{#1}}
\newlength\celldim
\newcolumntype{C}{@{}>{\centering\arraybackslash}p{\celldim} @{}}

\newcommand{\R}{\mathbb{R}}
\newcommand{\Hi}{\mathcal{H}}

\newcommand{\Id}{\operatorname{Id}}
\newcommand{\vect}{\operatorname{vec}}
\DeclareMathOperator{\argmin}{argmin}

\newtheorem{theorem}{Theorem}[section]
\newtheorem{proposition}{Proposition}[section]

\newtheorem{remark}{Remark}[section]

\renewenvironment*{proof}[1][Proof]{\noindent \textbf{#1.} }{\ \rule{0.5em}{0.5em}}
\numberwithin{equation}{section}

\newcounter{row}
\newcounter{col}

\newcommand\setrow[9]{
  \setcounter{col}{1}
  \foreach \n in {#1, #2, #3, #4, #5, #6, #7, #8, #9} {
    \edef\x{\value{col} - 0.5}
    \edef\y{9.5 - \value{row}}
    \node[anchor=center] at (\x, \y) {\n};
    \stepcounter{col}
  }
  \stepcounter{row}
}

\begin{document}

\title{%
\vspace*{-2cm}%
Finding magic squares with the Douglas--Rachford algorithm%%
\vspace*{-0.2cm}%
}
\author{Francisco J. Arag\'on Artacho\thanks{%
Corresponding Author.} \\
Departamento de Matem\'aticas \\
Universidad de Alicante\\ \href{mailto:francisco.aragon@ua.es}{\nolinkurl{francisco.aragon@ua.es}}
\and Paula Segura Mart\'inez \\
Departamento de Estad\'istica e Investigaci\'on Operativa\\
Universidad de Valencia\\
\href{mailto:paulamaths@gmail.com}{\nolinkurl{paulamaths@gmail.com}}}
\date{%
\vspace*{-0.8cm}%
} \maketitle

\begin{abstract}

In this expository paper, we show how to use the Douglas--Rachford algorithm as a successful heuristic for finding magic squares. The Douglas--Rachford algorithm is an iterative projection method for solving feasibility problems. Although its convergence is only guaranteed in the convex setting, the algorithm has been successfully applied to a number of similar nonconvex problems, such as solving Sudoku puzzles. We present two formulations of the nonconvex feasibility problem of finding magic squares, which are inspired by those of Sudoku, and test the Douglas--Rachford algorithm on them.
\medskip



\end{abstract}

\section{Introduction}\label{sec:1}

A \emph{magic square} of order $n$ is an $n \times n$ matrix whose elements are distinct positive integers $1,2,\ldots,n^2$ arranged in such a way that the sum of the  $n$ numbers in any horizontal, vertical, or main diagonal line is always the same number, known as the \emph{magic constant}. Clearly, the magic constant is equal to $n(n^2+1)/2$, as it is the result of dividing by $n$ the sum of $1, 2, \ldots, n^2$. Hence, for magic squares of order $3$, $4$ and $5$, their respective magic constants are $15$, $34$ and $65$.
In \cref{fig:luoshu} we show the unique magic square of order $3$, except for rotations and reflections.

\begin{figure}[ht!]\centering
\begin{minipage}{.45\textwidth}
\begin{center}
\begin{TAB}(e,0.75cm,0.75cm){|c|c|c|}{|c|c|c|}
4 & 9 & 2\\
3 & 5 & 7\\
8 & 1 & 6\\
\end{TAB}
\end{center}
\end{minipage}
\begin{minipage}{.45\textwidth}
\includegraphics[width=0.6\textwidth]{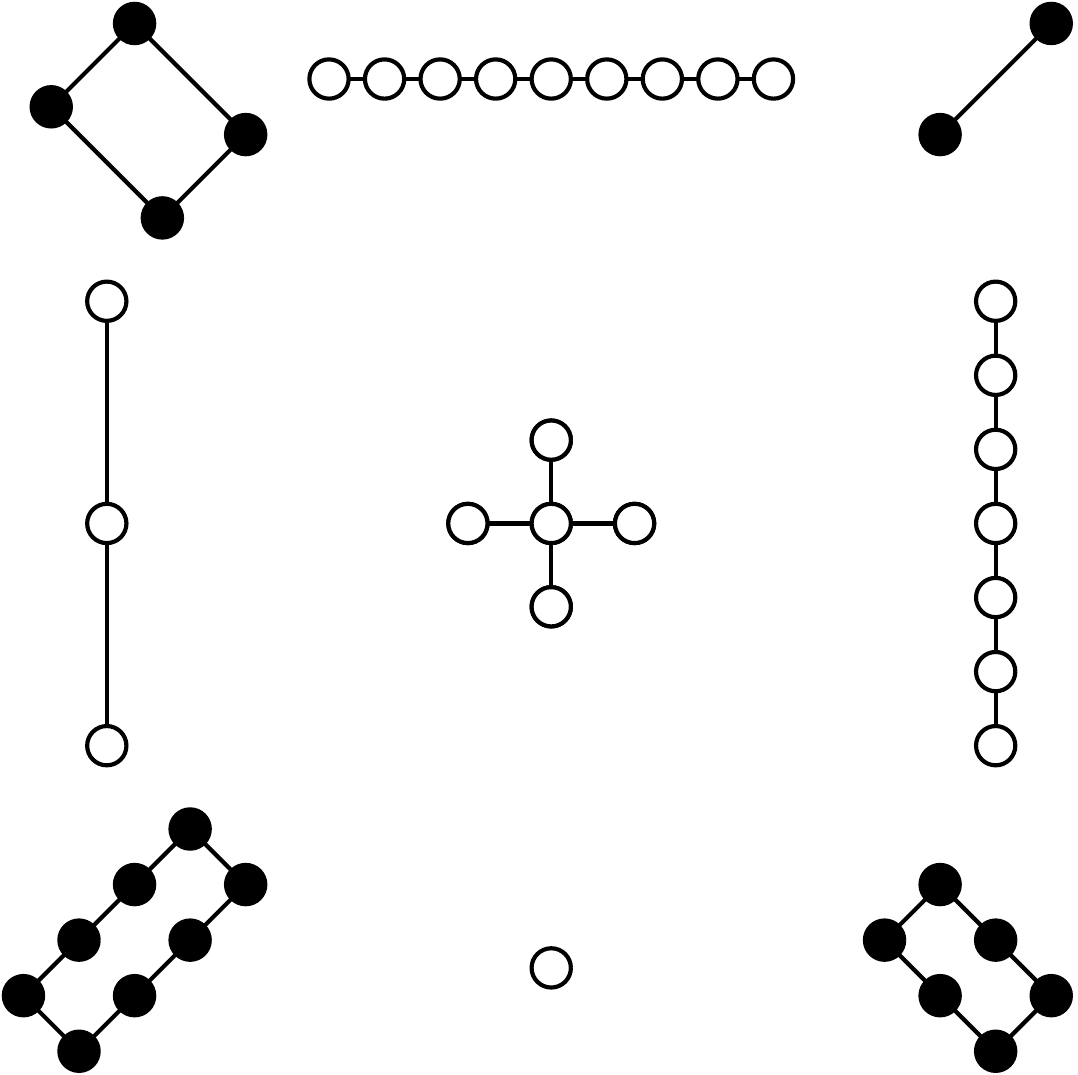}
\end{minipage}\vspace{5pt}
\caption{Modern and traditional representations of the \emph{luoshu} magic square\label{fig:luoshu}}
\end{figure}

\newlength{\lenquote}
\setlength{\lenquote}{.77\textwidth}

Magic squares have a long history and were frequently associated with mystical and supernatural properties. It is believed that the magic square of order~$3$ shown in \cref{fig:luoshu}, commonly referred by its Chinese name \emph{luoshu}, was known to Chinese mathematicians long before 500 BCE. According to~\cite{S08}, the very first textual reference to the \emph{luoshu} appears to be in the writings of Zhuang Zi (369-286  B.C.E.), one of the founders of Daoism. The name \emph{luoshu} means ``Luo River writing'', and refers to the following legend:
\begin{center}
\begin{minipage}{\lenquote}\emph{There was a huge flood in the ancient China. While Sage King Yu stood on the banks of the Luo River trying to channel the water out to the sea, a turtle with a curious pattern of dots arranged on its shell emerged from the river. This pattern was the magic square of order 3. Thereafter people were able to use this pattern in a certain way to control the river and protect themselves from floods.}
\end{minipage}
\end{center}

It is possible to construct magic squares of any order, except for $n=2$: if the following grid was a magic square,
\begin{center}
\begin{TAB}(e,0.6cm,0.6cm){|c|c|}{|c|c|}
a & b \\
c & d
\end{TAB}
\end{center}
then one should have $a+b=a+c$, which implies $b=c$.

Probably, the most well-known magic square of order $4$ is the one immortalized by the German artist Albrecht D\"urer in his engraving Melencolia I, see the detail in \cref{fig:Melencolia}. Observe how the date of the engraving, 1514, is shown in the two middle cells of the bottom row. This magic square has the additional property that the sums in any of the four quadrants, as well as the sum of the middle four numbers, are all $34$ (which is called a \emph{gnomon magic square}).

\begin{figure}[h]
\centering
\includegraphics[width=.85\textwidth]{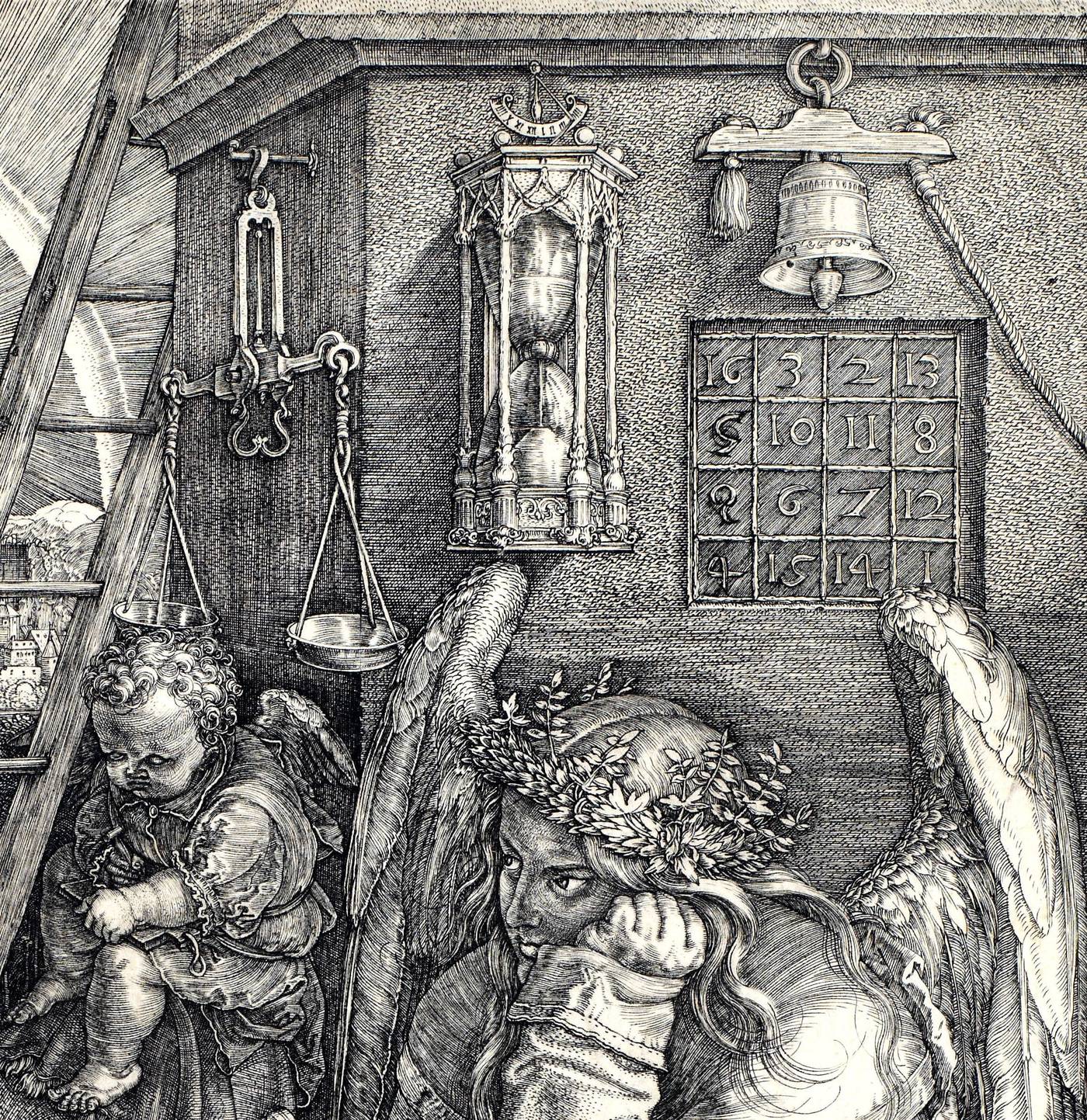}\vspace{5pt}
\caption{Albrecht D\"urer, Melencolia I (detail)}\label{fig:Melencolia}
\end{figure}

D\"urer became very interested in mathematics and devoted his life to the study of linear perspective and proportion. He claimed:
\begin{center}
\begin{minipage}{\lenquote}
\emph{The new art must be based upon science -- in particular, upon mathematics, as the most exact, logical, and graphically constructive of the sciences.}
\end{minipage}
\end{center}

D\"urer's magic square previously appeared in the manuscript \emph{De viribus quantitatis} by the Franciscan friar Luca Pacioli, which includes a short chapter on magic squares dating from about 1501--1503. In that chapter, Pacioli includes one example each of magic square of orders $3$ to $9$, and the one of order $4$ is identical to the one in Melancolia~I. Note that the chances for coincidence are very low, as there are $880$ magic squares of order $4$ (multiplied by their rotations and reflections). It is very plausible that D\"urer learnt this magic square from Pacioli, during his second visit to Italy between 1505--1507. For more details, see~\cite{S08,T03}.

There are many ways of constructing magic squares. In this work, based on~\cite{TFG}, we present a novel approach by formulating their construction as a nonconvex feasibility problem which can be tackled by the Douglas--Rachford algorithm. Of course, this only an expository problem, as there are much more efficient strategies for constructing magic squares. We review the main characteristics of this algorithm in \cref{sec:2}. Two different formulations are presented in \cref{sec:3}. The numerical experiments reported in \cref{sec:num_exp} demonstrate that the Douglas--Rachford algorithm can be used as a heuristic for finding them. We finish with some conclusions in~\cref{sec:5}.

\section{Projection algorithms}\label{sec:2}
Given a finite family $C_1,C_2,\ldots,C_N$ of subsets of a Hilbert space $\mathcal{H}$, a \emph{feasibility problem} consists in finding a point in their intersection; that is,
\begin{equation}\label{eq:FP}
\text{Find }x\in\bigcap_{i=1}^N C_i.
\end{equation}
Thanks to the generality of~\eqref{eq:FP}, plenty of problems can be modeled as feasibility problems. In many cases, finding a point in the intersection of the sets is a challenging task, but computing the projection of points into each of the sets is easy. In such scenarios, \emph{projection methods} become very useful. These are algorithms that use the projection of points onto the sets to define the iterates. Recall that the \emph{projection} of a point $x$ onto a closed set $C\subset\Hi$ is defined by
\begin{equation*}
P_C(x):= \left\{ z \in C : \| x-z \| = \inf_{c \in C} \| x-c \| \right\},
\end{equation*}
where $\|\cdot\|$ denotes the norm induced by the inner product $\langle\cdot,\cdot\rangle$ of $\Hi$. Observe that, in general, the projection operator $P_C:\Hi\rightrightarrows C$ is a set-valued mapping, see \cref{fig:projection}. Nonetheless, if the set $C$ is nonempty, closed and convex, the projection is uniquely determined as follows (see, e.g., \cite[Theorem~3.16]{BC17}):
\begin{equation}\label{eq:projection}
p=P_C(x)\iff p\in C \text{ and }\langle c-p,x-p\rangle\leq 0,\, \forall c\in C.
\end{equation}

\begin{figure}[ht!]
\centering
\includegraphics[width=0.45\textwidth]{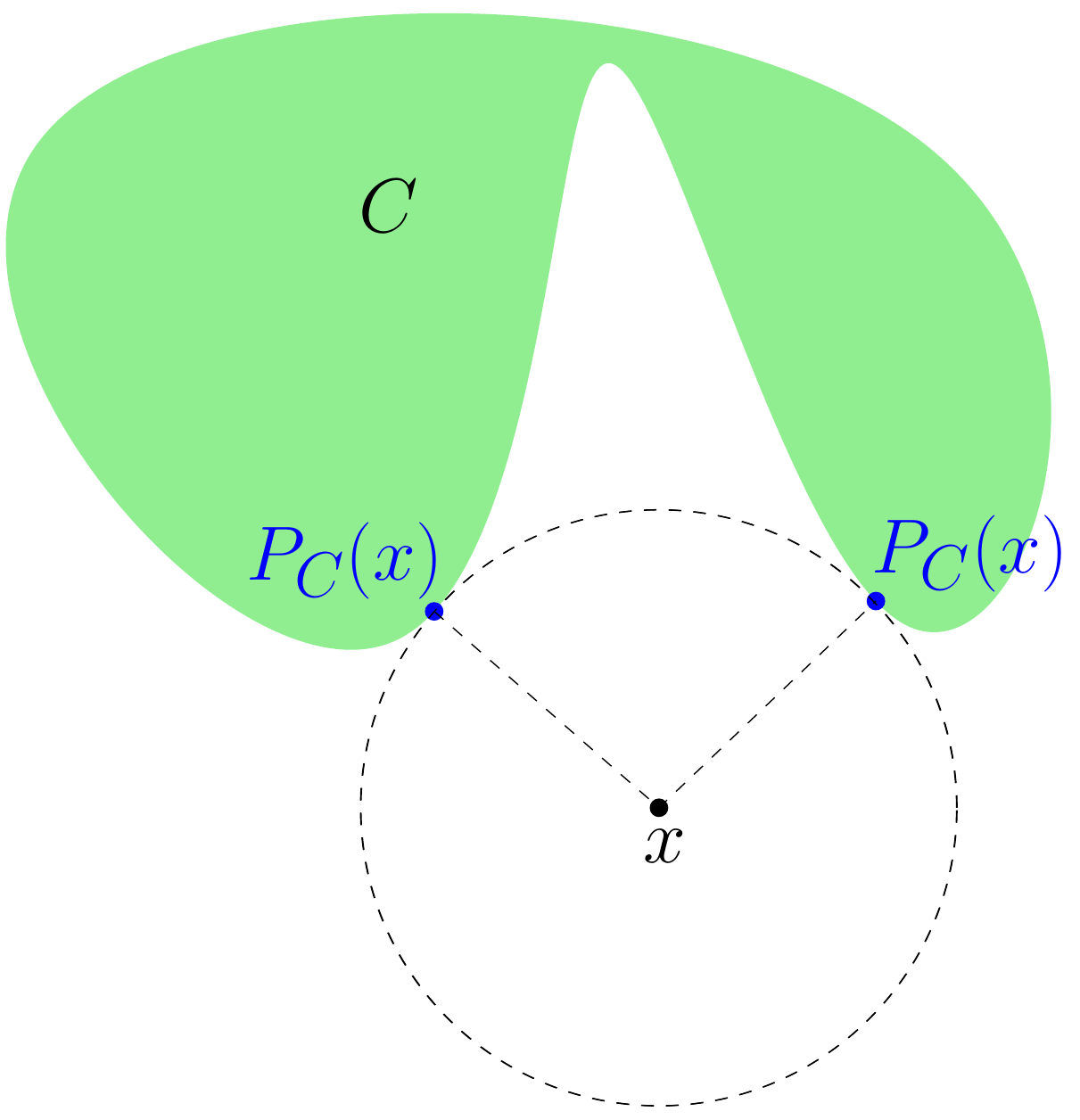}\vspace{3pt}
\caption{Projection of a point $x$ onto a nonconvex set $C$}\label{fig:projection}
\end{figure}

The most intuitive and well-known projection method is the von Neumann's alternating projection algorithm, which iterates by cyclically projecting onto each of the sets, see \cref{fig:AP}.
\begin{figure}[ht!]
\centering
\includegraphics[width=0.5\textwidth]{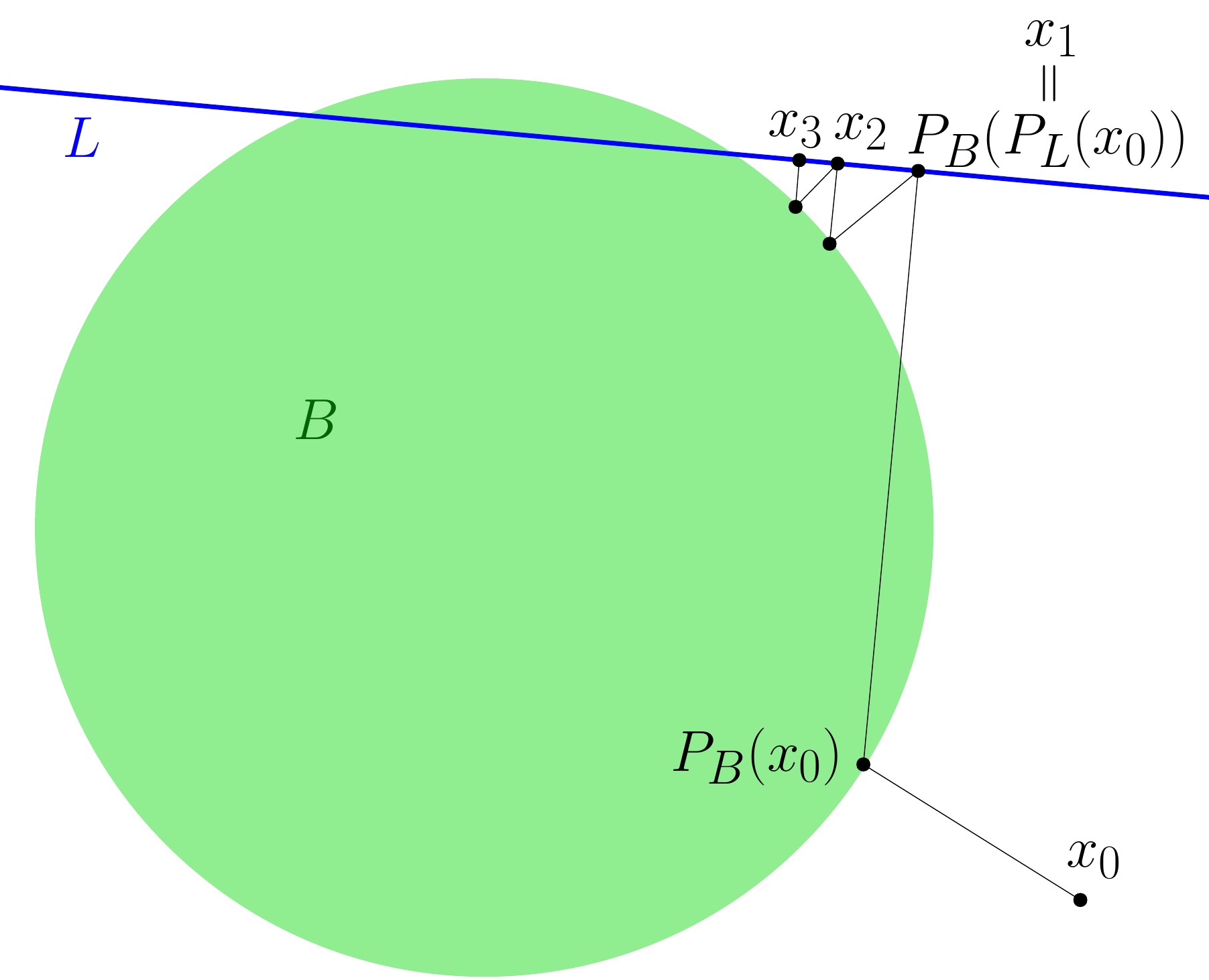}\vspace{5pt}
\caption{Alternating projection method applied to a ball and a line in $\R^2$}\label{fig:AP}
\end{figure}

In this work, we will focus on another popular projection method: the \emph{Douglas--Rachford algorithm}. Although the scheme has its origins in the work of J.~Douglas
and H.H.~Rachford~\cite{DR56}, where it was proposed for solving a system of linear equations arising in heat conduction problems, it was P.L. Lions and B.~Mercier~\cite{LM79} who really deserve the credit for the algorithm. In their work, not only they showed how to successfully generalize the algorithm for solving convex feasibility problems, but they actually provided a splitting algorithm for finding a zero in the sum of two maximally monotone operators.

The Douglas--Rachford algorithm employs the reflection operators when is applied for solving feasibility problems. Recall that the \emph{reflection} of a point $x \in \mathcal{H}$ onto a closed set $C$ is defined as $R_C(x):=2P_C(x)-x$. The iterates of the algorithm for solving feasibility problems involving two sets are obtained  by computing an average between the current point and two consecutive reflections onto the sets, see \cref{fig:DR}. For this reason, the scheme is also referred as the \emph{averaged alternating reflections method}.
\begin{figure}[ht!]
\centering
\includegraphics[width=0.6\textwidth]{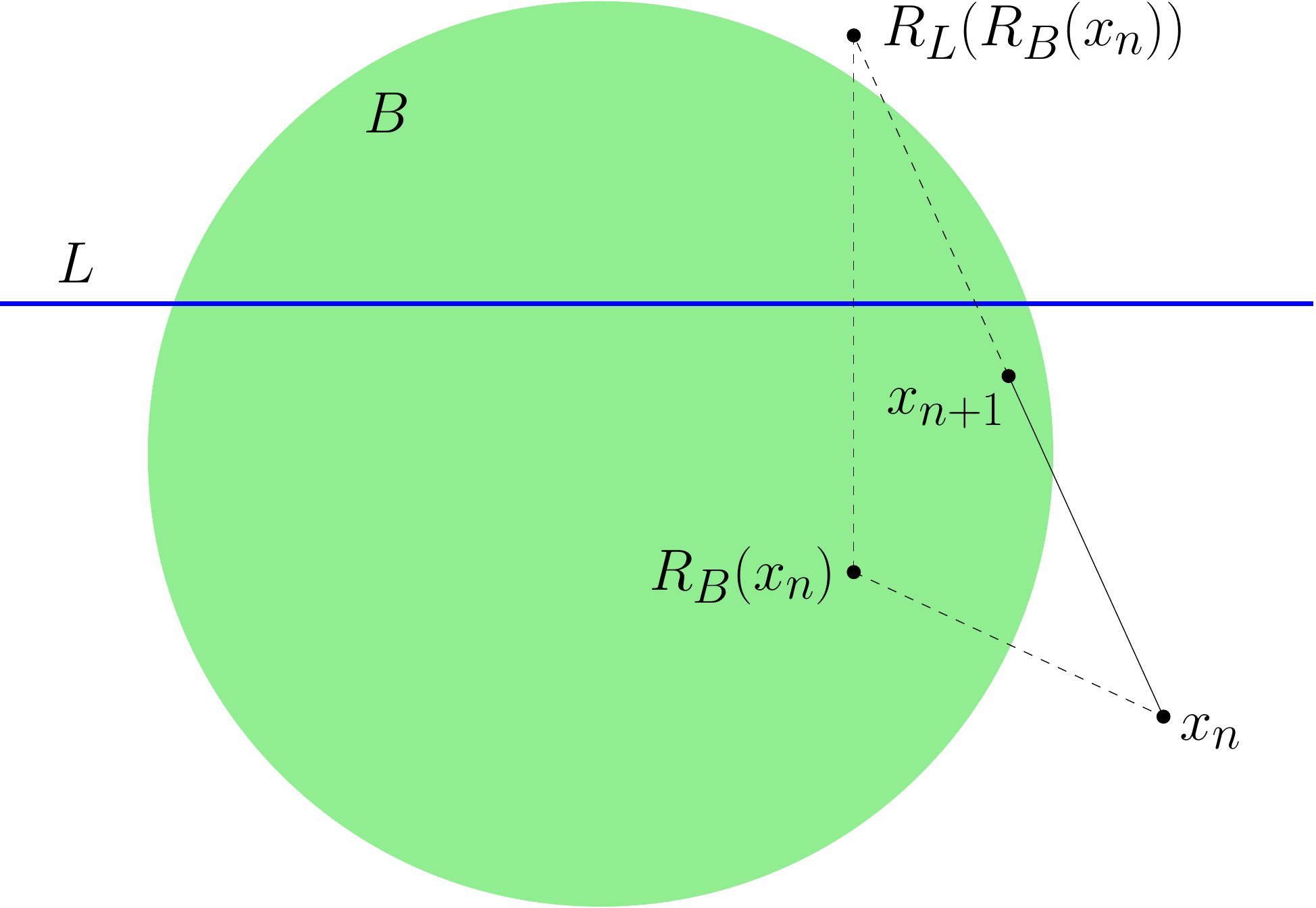}\vspace{5pt}
\caption{Douglas--Rachford method applied to a ball and a line in $\R^2$}\label{fig:DR}
\end{figure}

The Douglas--Rachford algorithm has recently gained great attention, in part thanks to its good performance as a heuristic in nonconvex settings (where the convergence of the algorithm is not guaranteed). In the remainder of this section we will review its main properties in the convex setting, we will show how the algorithm can be applied to any number of sets thanks to a product space reformulation of the feasibility problem, and we will review an application in the nonconvex setting for solving Sudoku puzzles.

\subsection{The Douglas--Rachford algorithm for convex sets}

When the Douglas--Rachford algorithm is applied to two closed and convex sets $A,B\subset \Hi$, it can be viewed as a fixed point iterative method defined by the \emph{Douglas--Rachford operator} $T_{A,B}:=\frac{1}{2}(\Id+R_B R_A)$, where $\Id$ denotes the identity mapping:
\begin{equation*} \label{eq1}
x_{n+1} = T_{A,B}(x_n)\quad\text{for }n=0,1,2,\ldots.
\end{equation*}
Note that the iterates are uniquely determined because the reflection operators onto convex sets are single-valued, as the projection operators are. Further, observe that $\overline{x}\in\Hi$ is a fixed point of the operator $T_{A,B}$ if and only if
\[
\overline{x} = 2P_B(2P_A(\overline{x})-\overline{x})-(2P_A(\overline{x})+\overline{x}) \iff P_A(\overline{x})=P_B(2P_A(\overline{x})-\overline{x}),
\]
which implies that $P_A(\overline{x}) \in A \cap B$. Therefore, the projection onto the set $A$ of any fixed point of the Douglas--Rachford operator solves the feasibility problem.

Using a well-known theorem of Opial~\cite[Theorem~1]{O67}, the sequence generated by the algorithm can be proved to be weakly convergent to a fixed point of the Douglas--Rachford operator, whenever $A\cap B\neq\emptyset$. This is a consequence of the \emph{firm nonexpansivity} of the projection operator,
\[
\| P_A (x) - P_A (y) \|^2 + \|(\Id-P_A)(x) - (\Id-P_A)(y)\|^2 \leq \| x-y \|^2, \quad\forall x,y \in \mathcal{H},
\]
which implies that the reflection operator is \emph{nonexpansive} (or Lipschitz continuous with constant $1$), i.e.,
$$\| R_A(x) - R_A(y) \| \leq \| x-y \| , \quad\forall x,y \in \mathcal{H},$$
and therefore, $T_{A,B}$ is firmly nonexpansive \cite[Theorem~12.1]{S2}.

The next result summarizes the main properties of the Douglas--Rachford algorithm in the convex setting.

\begin{theorem}\label{th:DR}
Let $A,B \in \mathcal{H}$ be closed and convex sets. Given any $x_0 \in \mathcal{H}$, define
$x_{n+1}=T_{A,B}(x_n)$, for every $n\geq 0$. Then, the following holds:
\begin{itemize}
\item [(i)] If $A \cap B \neq \emptyset$, then $\{x_n\}$ is weakly convergent to a point $x^\star$ and $\{P_A(x_n)\}$ is weakly convergent to $P_A(x^\star) \in A \cap B$,
\item [(ii)] If $A \cap B = \emptyset$, then $\| x_n \| \rightarrow \infty$.
\end{itemize}
\end{theorem}
\begin{proof}
 (i) For the first part, see~\cite[Theorem~3.13 and Corollary~3.9]{BCL04}. Note that the weak convergence of the \emph{shadow sequence} $\{P_A(x_n)\}$ cannot be derived from the weak convergence of $\{x_n\}$, since $P_A$ may not be weakly continuous (see~\cite[Example 4.20]{BC17}). The weak convergence of the shadow sequence was proved by Svaiter in~\cite[Theorem~1]{S11}. (ii) See~\cite[Corollary~2.2]{BBR78}.
\end{proof}

In contrast with the alternating projection algorithm, whose scheme can be easily generalized to more than two sets by cyclically projecting onto the sets, the same cannot be done with the Douglas--Rachford algorithm (see~\cite[Example~2.1]{ABT14} for a simple example involving three lines in $\R^2$). Fortunately, the product space formulation, due to Pierra~\cite{P84}, permits to reformulate any feasibility problem as another one involving only two sets. This is the subject of the next subsection.

\subsection{Product space formulation}\label{sec:product_space}

Consider the feasibility problem~\eqref{eq:FP} given by $N$ sets. Let
$$\mathcal{H}^N:=\mathcal{H} \times \overset{(N)}{\cdots} \times \mathcal{H} = \lbrace \mathbf{x} = (x_1, \ldots, x_N ) : x_i \in \mathcal{H}, i = 1, \ldots, N \rbrace.$$
It can be easily to checked that the space $\Hi^N$ endowed with the inner product $\langle \mathbf{x},\mathbf{y} \rangle := \sum_{i=1}^N \langle x_i, y_i \rangle$ is a Hilbert space. Let us consider two sets $C, D \subset \mathcal{H}^N$ defined by
$$C:=C_1\times C_2,\times\cdots\times C_N\quad\text{and}\quad D:=\left\lbrace (x,x,\ldots,x) \in \mathcal{H}^N : x \in \mathcal{H} \right\rbrace.$$
The set $D$, which is sometimes called the \textit{diagonal}, is always a closed subspace. Observe that the feasibility problem~\eqref{eq:FP} is equivalent to the one involving the sets $C$ and $D$, since
$$x \in \bigcap_{i=1}^N C_i \Longleftrightarrow (x,x,\ldots,x) \in C \cap D. $$

Not only~\eqref{eq:FP} can be recast as a feasibility problem involving two sets, but also the projection onto each of these sets has a closed form in terms of $P_{C_1}, P_{C_2}, \ldots, P_{C_N}$, as shown next. For completeness, we include a proof based on~\cite[Proposición 3.1]{ABT14b}. This result extends Pierra's Lemma~1.1 in~\cite{P84} to possibly nonconvex sets.

\begin{proposition}
For any $\mathbf{x}=(x_1,x_2,\ldots,x_N ) \in \mathcal{H}^N$, one has
$$P_D (\mathbf{x}) = \left( \frac{1}{N} \sum_{i=1}^N x_i,\ldots,\frac{1}{N} \sum_{i=1}^N x_i \right), $$
and if $P_{C_1}(x_1),P_{C_2}(x_2)\ldots,P_{C_N}(x_N)$ are nonempty,
$$P_C (\mathbf{x}) = P_{C_1}(x_1)\times P_{C_2}(x_2)\times\cdots\times P_{C_N}(x_N).$$
\end{proposition}

\begin{proof}
Let $\mathbf{p}=(p,\ldots,p) \in D$ be the projection of $\mathbf{x}$ onto the set $D$. For any $d \in \mathcal{H}$, one has $\mathbf{d}:=(d,\ldots,d) \in D$. Then, using the characterization~\eqref{eq:projection} applied to $\mathbf{p}+\mathbf{d}$ and $\mathbf{p}-\mathbf{d}$, we deduce
$$0=\langle \mathbf{d},\mathbf{x}-\mathbf{p}\rangle = \sum_{i=1}^N \langle d,x_i-p \rangle = \left\langle d,\sum_{i=1}^N x_i - Np\right\rangle.$$
As the latter holds for any $d\in\Hi$, it must be $p=\frac{1}{N} \sum_{i=1}^N x_i$, as claimed.

To prove the formula for the projection onto the set $C$, pick first any point $\mathbf{p}=(p_1,\ldots,p_N) \in P_{C_1}(x_1)\times\cdots\times P_{C_N}(x_N) \subseteq C$. Then, for any $\mathbf{c}=(c_1,\ldots,c_N) \in C$, one has
$$\| \mathbf{x}-\mathbf{c} \|^2 = \sum_{i=1}^N \| x_i - c_i \|^2 \geq \sum_{i=1}^N \| x_i - p_i \|^2 = \| \mathbf{x}-\mathbf{p} \|^2.$$
Therefore, $P_C(\mathbf{x}) \supseteq \prod_{i=1}^N P_{C_i}(x_i)$.

Conversely, let $\mathbf{p}=(p_1,\ldots,p_N)\in P_C(\mathbf{x})$. Pick any $j\in\{1,\ldots,N\}$ and choose any $c_j\in C_j$. Let $q_i:=p_i$ if $i \neq j$ and let $q_j:=c_j$, and define $\mathbf{q}:=(q_1,\ldots,q_N) \in C$. Then,
$$\sum_{i=1}^N \| x_i - p_i \|^2=\| \mathbf{x}-\mathbf{p} \|^2  \leq \| \mathbf{x}-\mathbf{q} \|^2=\sum_{i=1}^N \| x_i - q_i \|^2=\sum_{i=1\atop i\neq j}^N \| x_i - p_i \|^2+\|x_j-c_j\|^2,$$
whence, $\|x_j-p_j\|\leq \|x_j-c_j\|$. Since $c_j$ was arbitrarily chosen in $C_j$, this implies that $p_j\in P_{C_j}(x_j)$, and concludes the proof.
\end{proof}

Pierra's product space formulation works well in practice whenever the number of sets $N$ is small. When $N$ is large, the computational cost of calculating the iterates in the large space $\Hi^N$ instead of $\Hi$ becomes prohibitive. In these cases, the cyclic Douglas--Rachford algorithm with $r$-sets-Douglas--Rachford operators, recently introduced in~\cite{ACG18}, would be a better choice. The latter is a generalization of the cyclic Douglas--Rachford method of Borwein--Tam~\cite{BT14}, whose iterates are defined by
$$x_{n+1}=T_{C_N,C_1}T_{C_{N-1},C_N}\cdots T_{C_2,C_3}T_{C_1,C_2}(x_n);$$
that is, a cyclic composition of Douglas--Rachford operators applied to the sets.

\subsection{The Douglas--Rachford algorithm for nonconvex sets} \label{sec:sud}

\cref{th:DR} only guarantees the global convergence of the Douglas--Rachford algorithm for convex sets. In spite of this, the algorithm has been successfully applied as a heuristic for solving many nonconvex problems, particularly those of combinatorial type (see, e.g.,~\cite{ABT14,ABT14b,AC18,BKroad,Elser}). In most applications in the nonconvex setting, the constraint sets satisfy some type of regularity property and local convergence can be proved~\cite{BNlocal,HLnonconvex,Plinear}. Nonetheless, the results on global behavior are limited to very special sets~\cite{ABT16,B15} and cannot explain the good performance of the algorithm in the nonconvex setting.

It is clear that the performance of an algorithm heavily depends on how the problem is formulated. When the Douglas--Rachford algorithm is applied to a nonconvex problem, finding an ``adequate'' formulation can be crucial for its success as a heuristic. Indeed, the formulation of the problem often determines whether or not the algorithm can successfully solve the problem at hand always, frequently or never.

Let us briefly present an interesting example exhibiting this behavior: Sudoku puzzles. Their solution with the Douglas--Rachford algorithm was first proposed in~\cite{Elser}, and subsequently analyzed in~\cite{S8} and \cite{ABT14b}. A \emph{Sudoku} is a $9 \times 9$ grid, divided in nine $3 \times 3$ subgrids, with some entries already prefilled. The objective is to fill the remaining entries in such a way that each row, column and subgrid contains the digits from $1$ to $9$ exactly once, see \cref{fig:Sudoku}.

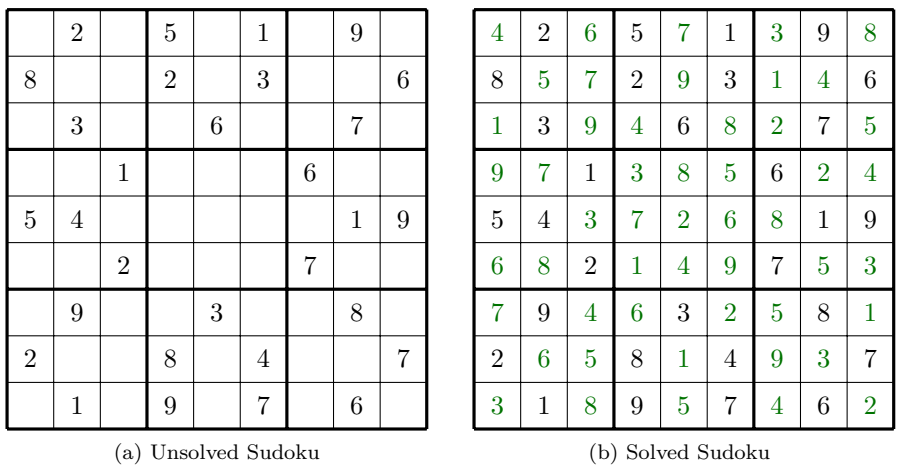
\begin{figure}[ht!]\centering
\subfloat[Unsolved Sudoku]{
\begin{tikzpicture}[scale=.62]

  \begin{scope}
    \draw (0, 0) grid (9, 9);
    \draw[very thick, scale=3] (0, 0) grid (3, 3);

    \setcounter{row}{1}
    \setrow { }{2}{ }  {5}{ }{1}  { }{9}{ }
    \setrow {8}{ }{ }  {2}{ }{3}  { }{ }{6}
    \setrow { }{3}{ }  { }{6}{ }  { }{7}{ }

    \setrow { }{ }{1}  { }{ }{ }  {6}{ }{ }
    \setrow {5}{4}{ }  { }{ }{ }  { }{1}{9}
    \setrow { }{ }{2}  { }{ }{ }  {7}{ }{ }

    \setrow { }{9}{ }  { }{3}{ }  { }{8}{ }
    \setrow {2}{ }{ }  {8}{ }{4}  { }{ }{7}
    \setrow { }{1}{ }  {9}{ }{7}  { }{6}{ }
  \end{scope}
\end{tikzpicture}
  }\quad
\subfloat[Solved Sudoku]{
\begin{tikzpicture}[scale=.62]
  \begin{scope}[xshift=12cm]
    \draw (0, 0) grid (9, 9);
    \draw[very thick, scale=3] (0, 0) grid (3, 3);

    \setcounter{row}{1}
    \setrow { }{2}{ }  {5}{ }{1}  { }{9}{ }
    \setrow {8}{ }{ }  {2}{ }{3}  { }{ }{6}
    \setrow { }{3}{ }  { }{6}{ }  { }{7}{ }

    \setrow { }{ }{1}  { }{ }{ }  {6}{ }{ }
    \setrow {5}{4}{ }  { }{ }{ }  { }{1}{9}
    \setrow { }{ }{2}  { }{ }{ }  {7}{ }{ }

    \setrow { }{9}{ }  { }{3}{ }  { }{8}{ }
    \setrow {2}{ }{ }  {8}{ }{4}  { }{ }{7}
    \setrow { }{1}{ }  {9}{ }{7}  { }{6}{ }

    \begin{scope}[green!45!black]
      \setcounter{row}{1}
      \setrow {4}{ }{6}  { }{7}{ }  {3}{ }{8}
      \setrow { }{5}{7}  { }{9}{ }  {1}{4}{ }
      \setrow {1}{ }{9}  {4}{ }{8}  {2}{ }{5}

      \setrow {9}{7}{ }  {3}{8}{5}  { }{2}{4}
      \setrow { }{ }{3}  {7}{2}{6}  {8}{ }{ }
      \setrow {6}{8}{ }  {1}{4}{9}  { }{5}{3}

      \setrow {7}{ }{4}  {6}{ }{2}  {5}{ }{1}
      \setrow { }{6}{5}  { }{1}{ }  {9}{3}{ }
      \setrow {3}{ }{8}  { }{5}{ }  {4}{ }{2}
    \end{scope}

  \end{scope}

\end{tikzpicture}}\vspace{5pt}
\caption{Example of Sudoku}\label{fig:Sudoku}
\end{figure}

Then, Sudokus are nothing else than matrix completion problems, which can be modeled as feasibility problems. There are different possibilities for choosing the constraint sets $C_1,\ldots,C_N$ in such a way that $\bigcap_{i=1}^N C_i$ coincides with the (unique) solution to the Sudoku. Nonetheless, in order to apply the Douglas--Rachford algorithm, the sets must be chosen in such a way that the projections can be efficiently computed, ideally having a closed form. In \cite[Sección 6]{ABT14b}, two ways of modeling Sudokus are analyzed: as integer and binary feasibility problems. We briefly describe them in what follows.

\subsubsection*{Sudoku modeled as an integer program}

It is obvious how to model a Sudoku as an integer feasibility problem.
Let $S$ be the partially filled $9\times 9$ matrix representing the unsolved Sudoku, let $J\subset\{1,2,\ldots,9\}^2$ be the set of indices for which $S$ is filled, and let $A_{i, j}$ denote the $(i, j)$-th entry of the matrix $A$. If we denote by $\mathcal{C}$ the set of vectors which are permutations of $1,2,\ldots,9$, then a matrix $A\in\mathbb{R}^{9\times 9}$ is a solution to the Sudoku if and only if
$A \in C_1 \cap C_2 \cap C_3 \cap C_4$, where
\begin{align*}
C_1 &:= \left\lbrace A\in\mathbb{R}^{9\times 9}: \text{each row of $A$ belongs to $\mathcal{C}$}\right\rbrace, \\
C_2 &:= \left\lbrace A\in\mathbb{R}^{9\times 9}: \text{each column of $A$ belongs to $\mathcal{C}$} \right\rbrace, \\
C_3 &:= \left\lbrace A\in\mathbb{R}^{9\times 9}: \text{each of the $9$ subgrids of $A$ belongs to $\mathcal{C}$}\right\rbrace, \\
C_4 &:= \left\lbrace A\in\mathbb{R}^{9\times 9}: A_{i,j}=S_{i,j} \text{ for each } (i,j) \in J \right\rbrace.
\end{align*}

The projection onto $C_4$ can be easily computed componentwise, while the projections onto $C_1$, $C_2$ and $C_3$ are determined by the next result, see~\cite{ABT14b} for more details.
\begin{proposition} \label{propo}
Denote by $\mathcal{C} \in \mathbb{R}^m$ the set of vectors whose entries are all permutations of $c_1, c_2, \ldots, c_m \in \mathbb{R}$. Then, for any $x \in \mathbb{R}^{m}$, one has
$$P_{\mathcal{C}} (x) = [\mathcal{C}]_x,$$
where $[\mathcal{C}]_x$ denotes the set of vectors in $\mathcal{C}$ such that $y \in [\mathcal{C}]_x$ if the $i$th largest entry of $y$ has the same index in $y$ as the $i$th largest entry of $x$.
\end{proposition}
\begin{proof}
See~\cite[Proposition~5.1]{ABT14b}.
\end{proof}

\subsubsection*{Sudoku modeled as a zero-one program}

To model Sudoku puzzles as binary programs, we reformulate a matrix $A \in \mathbb{R}^{9 \times 9}$ as $B \in \mathbb{R}^{9 \times 9 \times 9}$, where
$$
B_{i,j,k}:= \begin{cases}
1 & \text{if } A_{i,j} = k,\\
0 & \text{otherwise.}
\end{cases}
$$
This reformulation transforms the entries into a 3-dimensional zero-one array, which can be thought as a cube, see~\cref{fig:zero-one}.

\begin{figure}[ht!]
\begin{center}
\includegraphics[scale=0.65]{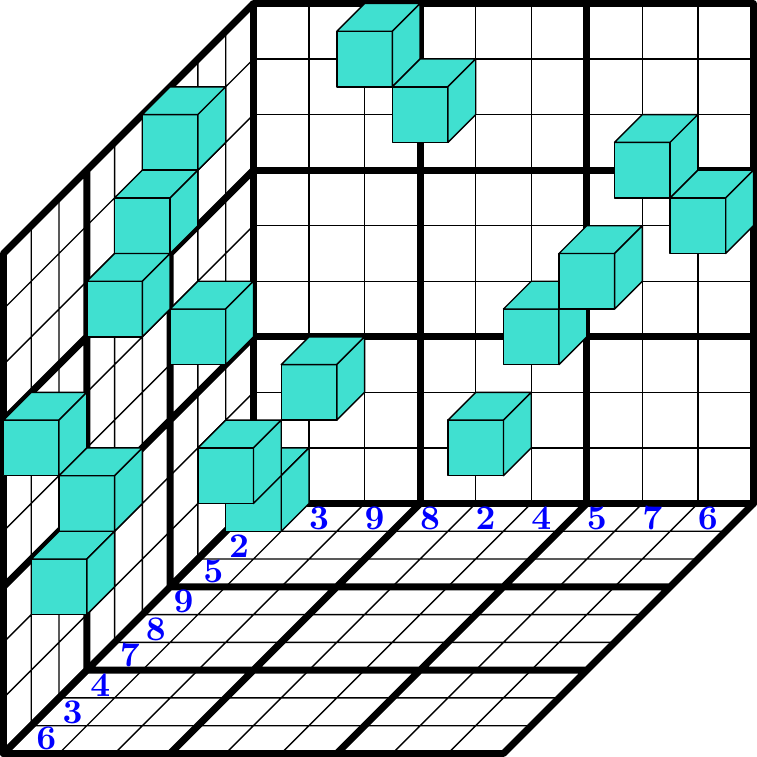}
\end{center}
\caption{Zero-one representation of a Sudoku puzzle}\label{fig:zero-one}
\end{figure}

Let us denote by $S'$ the $9 \times 9 \times 9$ zero-one array corresponding to the unsolved Sudoku $S$, let $I:=\{1,2,\ldots,9\}$, let $J' \subseteq I^3$ be the set of indices for which $S'$ is filled, and let $\mathcal{B}$ denote the $9$-dimensional standard basis. The four constraints of the integer program above become now
\begin{align*}
C_1 &:= \left\lbrace B\in\mathbb{R}^{9\times 9\times 9}: B_{i,:,k} \in \mathcal{B} \text{ for each } i,k \in I \right\rbrace, \\
C_2 &:= \left\lbrace B\in\mathbb{R}^{9\times 9\times 9}: B_{:,j,k} \in \mathcal{B} \text{ for each } j,k \in I \right\rbrace, \\
C_3 &:= \lbrace B\in\mathbb{R}^{9\times 9\times 9}: \vect  B_{3i+1:3(i+1),3j+1:3(j+1),k} \in \mathcal{B} \text{ for } i, j = 0,1,2 \text{ and } k \in I \rbrace, \\
C_4 &:= \left\lbrace B\in\mathbb{R}^{9\times 9\times 9}: B_{i,j,k}=1 \text{ for each } (i,j,k) \in J' \right\rbrace,
\end{align*}
where $\vect A$ represents the \emph{vectorization} of a matrix $A$ by columns.

Further, since each cell in the Sudoku can only have one number assigned, we must add the additional constraint
$$C_5 = \left\lbrace B\in\mathbb{R}^{9\times 9\times 9}: B_{i,j,:} \in \mathcal{B} \text{ for each } i,j \in I \right\rbrace.$$
These constraints are represented in \cref{fig:rest-sudoku}.

\begin{figure}[h]
\centering
\subfloat[$C_1$]{
\label{c1}
\includegraphics[width=0.22\textwidth]{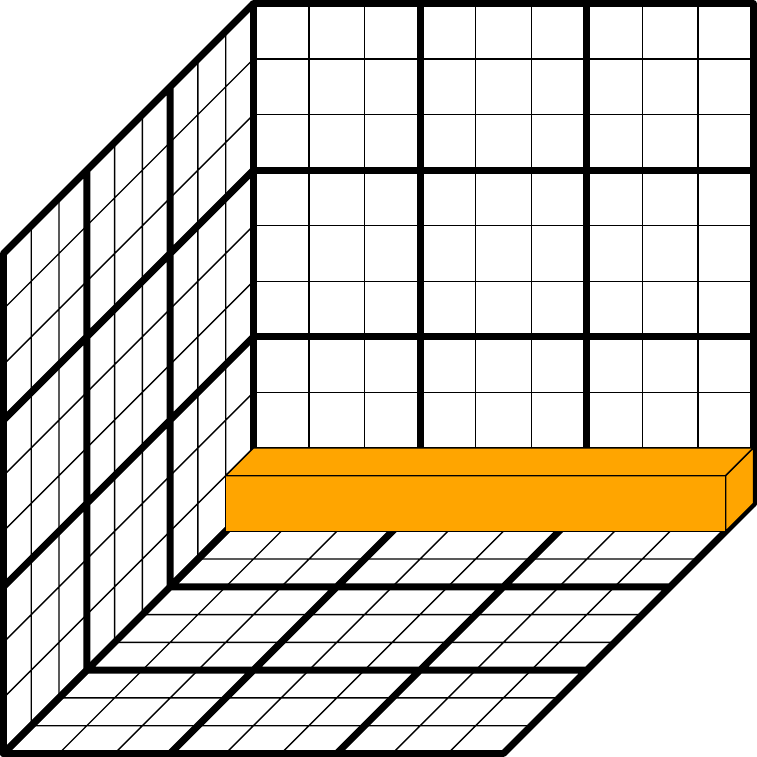}}
\hspace{1pt}
\subfloat[$C_2$]{
\label{c2}
\includegraphics[width=0.22\textwidth]{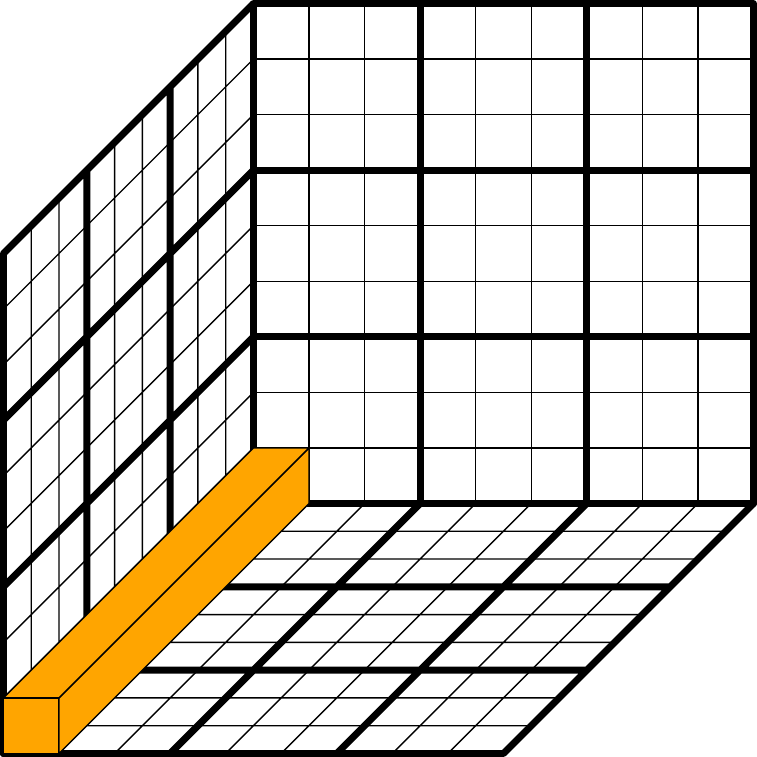}}
\hspace{1pt}
\subfloat[$C_3$]{
\label{c3}
\includegraphics[width=0.22\textwidth]{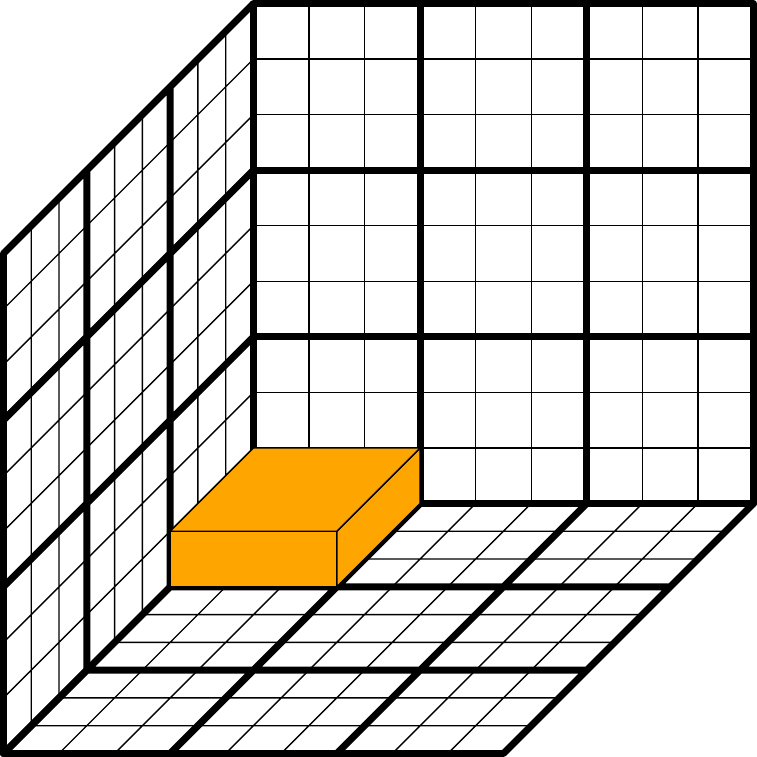}}
\hspace{1pt}
\subfloat[$C_5$]{
\label{c5}
\includegraphics[width=0.22\textwidth]{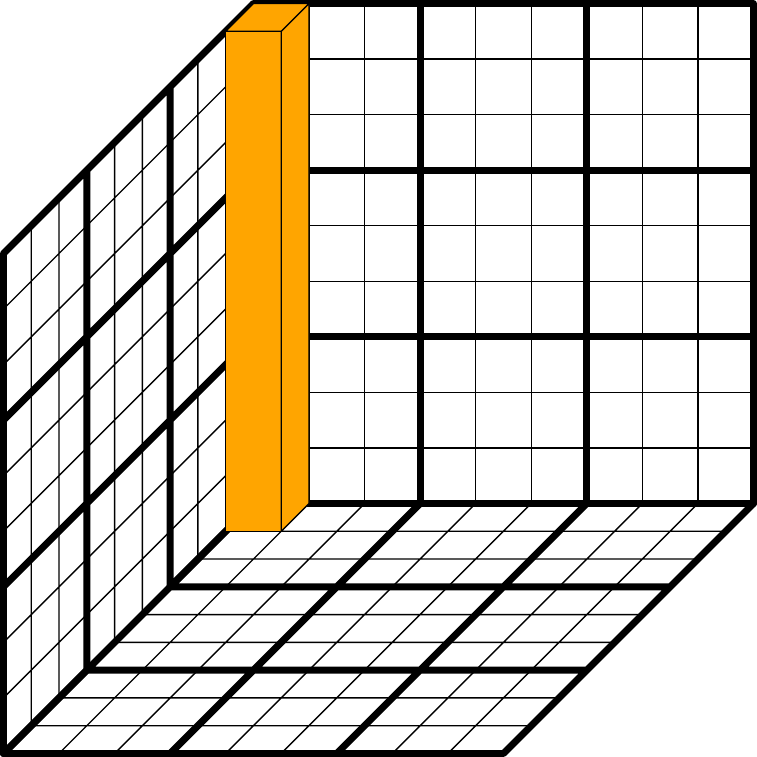}}
\vspace{2pt}
\caption{Visualization of the constraints used for modeling Sudoku as a zero-one program. The entries in the colored blocks are all $0$ except for a single~$1$}\label{fig:rest-sudoku}
\end{figure}

Then, $B$ completes $S'$ (and thus solves the Sudoku) if and only if $B \in C_1 \cap C_2 \cap C_3 \cap C_4 \cap C_5$. Again, the projection onto $C_4$ can be easily computed componentwise, while the projections onto $C_1$, $C_2$, $C_3$ and $C_5$ are determined by \cref{propo}.

\subsubsection*{Performance of the Douglas--Rachford algorithm on Sudokus}
As observed in~\cite{ABT14b}, the Douglas--Rachford algorithm is totally ineffective for solving the integer formulation. On the other hand, the algorithm is very successful when it is applied to the binary formulation, being able to solve nearly all the instances (without a restart) in all the Sudoku libraries tested in~\cite{ABT14b}. Even in the one where it was most unsuccessful, the \texttt{top95}\footnote{\texttt{top95}: \url{http://magictour.free.fr/top95}} library, it solved 87\% of the instances.

Another possibility for modeling Sudoku as feasibility problems was tested in~\cite{AC18,ACE19}, based on the fact that an unsolved Sudoku can be viewed as a graph precoloring problem. Surprisingly, the \emph{rank formulation} proposed in~\cite{ACE19} had a success rate of 100\% in the \texttt{top95} library, and no Sudoku has been found so far for which Douglas--Rachford fails to find its solution for any starting point.

Therefore, it is clear that the choice of the formulation has a big impact in the success rate of the Douglas--Rachford algorithm when it is applied in the nonconvex setting.

\section{Finding magic squares as a feasibility problem}\label{sec:3}

In this next section, we propose two different formulations for finding magic squares as a feasibility problem which are inspired by the binary and integer formulations of Sudokus. To formulate the search, the first thing we need to do is to choose some sets whose intersection includes all the properties that characterize a magic square of order $n$: the sum of each row, column and main diagonals must be equal to the magic constant $c:=\frac{n(n^2+1)}{2}$, and it must contain all the numbers between $1$ and $n^2$. Note that if our purpose was to complete a prefilled magic square, we would need to add an additional constraint that fixes the prefilled numbers.

The most obvious way to look for a magic square is to use an integer formulation. This is explained in the next subsection, where we also deduce the expression for the projection onto each of the constraint sets. Afterwards, we show how to formulate the search for magic squares as a zero-one program, in line with what we have seen for Sudokus in \cref{sec:sud}, although in a slightly different way.

\subsection{Magic squares modeled as an integer program}

If we denote by $\mathcal{P}$ the set of permutations of $1, 2 \ldots, n^2$ and $I:=\{1,2,\ldots,n\}$, then $A \in \mathbb{R}^{n \times n}$ is a magic square if and only if
$A\in\bigcap_{i=1}^5C_i$, %$A \in C_1 \cap C_2 \cap C_3 \cap C_4 \cap C_5,$
where
\begin{align*}
C_1 &:= \left\lbrace A \in \mathbb{R}^{n \times n} : \sum_{j \in I} A_{i,j} = c \text{ for each } i \in I \right\rbrace, \\
C_2 &:= \left\lbrace A \in \mathbb{R}^{n \times n} : \sum_{i \in I} A_{i,j} = c \text{ for each } j \in I \right\rbrace, \\
C_3 &:= \left\lbrace A \in \mathbb{R}^{n \times n} : \sum_{i \in I} A_{i,i} = c \right\rbrace, \\
C_4 &:= \left\lbrace A \in \mathbb{R}^{n \times n} : \sum_{i \in I} A_{i,n+1-i} = c \right\rbrace, \\
C_5 &:= \left \lbrace A \in \mathbb{R}^{n \times n} : \vect  A \in \mathcal{P} \right\rbrace.
\end{align*}

The first four sets are clearly convex, and their projection operators have a closed form determined by the next result.
\begin{proposition} \label{propi}
Consider $S=\left\lbrace x \in \mathbb{R}^m  : \sum_{i=1}^m x_i =c \right\rbrace$. For any $x \in \mathbb{R}^m$,
$$P_S(x)=x +\frac{1}{m} \left( c- \sum_{i=1}^m x_i \right)e, \quad \text{where } e = [1,1,\ldots,1]^T.$$
\end{proposition}

\begin{proof}
This follows from the standard formula for
the orthogonal projection onto a hyperplane, since $S=\left\lbrace x \in \mathbb{R}^m  : \langle x,e\rangle =c \right\rbrace$; see, e.g.,~\cite[Example~3.23]{BC17}.%\cite[Section 4.2.1]{ER11}.
\end{proof}\smallskip

Thereby, the projections onto each of these sets are given by
\begin{align*}
P_{C_1}(A) &= A + \frac{1}{n}\left( \begin{array}{ccc} c-\sum_{i=1}^n A_{1,i} & \cdots & c -\sum_{i=1}^n A_{1,i}\\
%c-\sum_{i=1}^n A_{2,i} & \cdots & c -\sum_{i=1}^n A_{2,i}\\
\vdots & & \vdots \\ c-\sum_{i=1}^n A_{n,i} & \cdots & c -\sum_{i=1}^n A_{n,i}
\end{array} \right),\\
P_{C_2}(A)&= A + \frac{1}{n}\left( \begin{array}{ccc} c-\sum_{i=1}^n A_{i,1} &  \cdots & c -\sum_{i=1}^n A_{i,n}\\
%c-\sum_{i=1}^n A_{i,1} & \cdots & c -\sum_{i=1}^n A_{i,n}\\
\vdots &  &  \vdots \\
c-\sum_{i=1}^n A_{i,1}  & \cdots & c -\sum_{i=1}^n A_{i,n}
\end{array} \right),\\
P_{C_3}(A)&= A + \frac{1}{n}\left(c-\sum_{i=1}^n A_{i,i}\right)I_n,\\
P_{C_4}(A)&= A + \frac{1}{n}\left(c-\sum_{i=1}^n A_{i,n+i-1}\right) \left( \begin{array}{cccc} 0 & \cdots & 0 & 1\\
\vdots &\udots & 1 &0\\
0 & \udots  &  \udots& \vdots\\
1 & 0 & \cdots &0
\end{array} \right),
\end{align*}
where $I_n$ denotes the identity matrix in $\mathbb{R}^{n\times n}$.

Finally, $C_5$ is a nonconvex set containing all those matrices whose entries are permutations of $ 1, 2, \ldots, n ^ 2 $, so its projection operator is determined by \cref{propo}.
Therefore, given an $ n \times n $ matrix, to compute its projection onto $C_5$, we just need to sort the elements of the matrix in ascending order and place the number $1$ in the cell that contains the smallest number, the number $2$ in the cell that contains the next one, and so on. If there are two equal elements, the projection of the matrix is not unique (in our experiments in \cref{sec:num_exp}, we chose the element that appeared first when the matrix was read by rows).

\subsubsection*{Completing a partially filled magic square}

If $M$ is a partially complete matrix representing an incomplete magic square, denote by $J \subseteq I^2$ the set of indices for which $M$ is filled. To find a completion of the magic square, we need to add the constraint
$$C_6 := \left\lbrace A \in \mathbb{R}^{n \times n} : A_{i,j} = M_{i,j} \text{ for each } (i,j) \in J \right\rbrace $$
to the set of constraints that we already have. Then, $A$ completes $M$ if, and only if, $$A \in C_1 \cap C_2 \cap C_3 \cap C_4 \cap C_5 \cap C_6.$$
The projection onto $C_6$ is given componentwise by
$$
P_{C_6} (A_{i,j}) = \begin{cases}
M_{i,j} & \text{if } (i,j) \in J,\\
A_{i,j} & \text{otherwise, }
\end{cases}
$$
for each $(i,j) \in I^2$.

\subsection{Magic squares modeled as a binary program}

To model the search for a magic square as a binary feasibility problem, we reformulate a matrix $A \in \mathbb{R}^{n\times n}$ as $B \in \mathbb{R}^{n\times n \times n^2}$, where
$$
B_{i,j,k}:= \begin{cases}
1 & \text{if } A_{i,j} \geq k,\\
0 & \text{otherwise.}
\end{cases}
$$
In this way, we transform the entries of the matrix into a $3$-dimensional zero-one array, and each number in the magic square can be thought as a \emph{pillar} made of small cubes, see Fig. \ref{fig:bin}.

\begin{figure}%[ht!]
	\centering
	\begin{minipage}{.45\textwidth}
		\begin{center}
			\begin{TAB}(e,0.75cm,0.75cm){|c|c|c|}{|c|c|c|}
				\textbf{2} & \textbf{7} & \textbf{6}\\
				9 & 5 & 1\\
				4 & 3 & 8\\
			\end{TAB}
		\end{center}
	\end{minipage}
	\begin{minipage}{.45\textwidth}
		\includegraphics[width=0.25\textwidth]{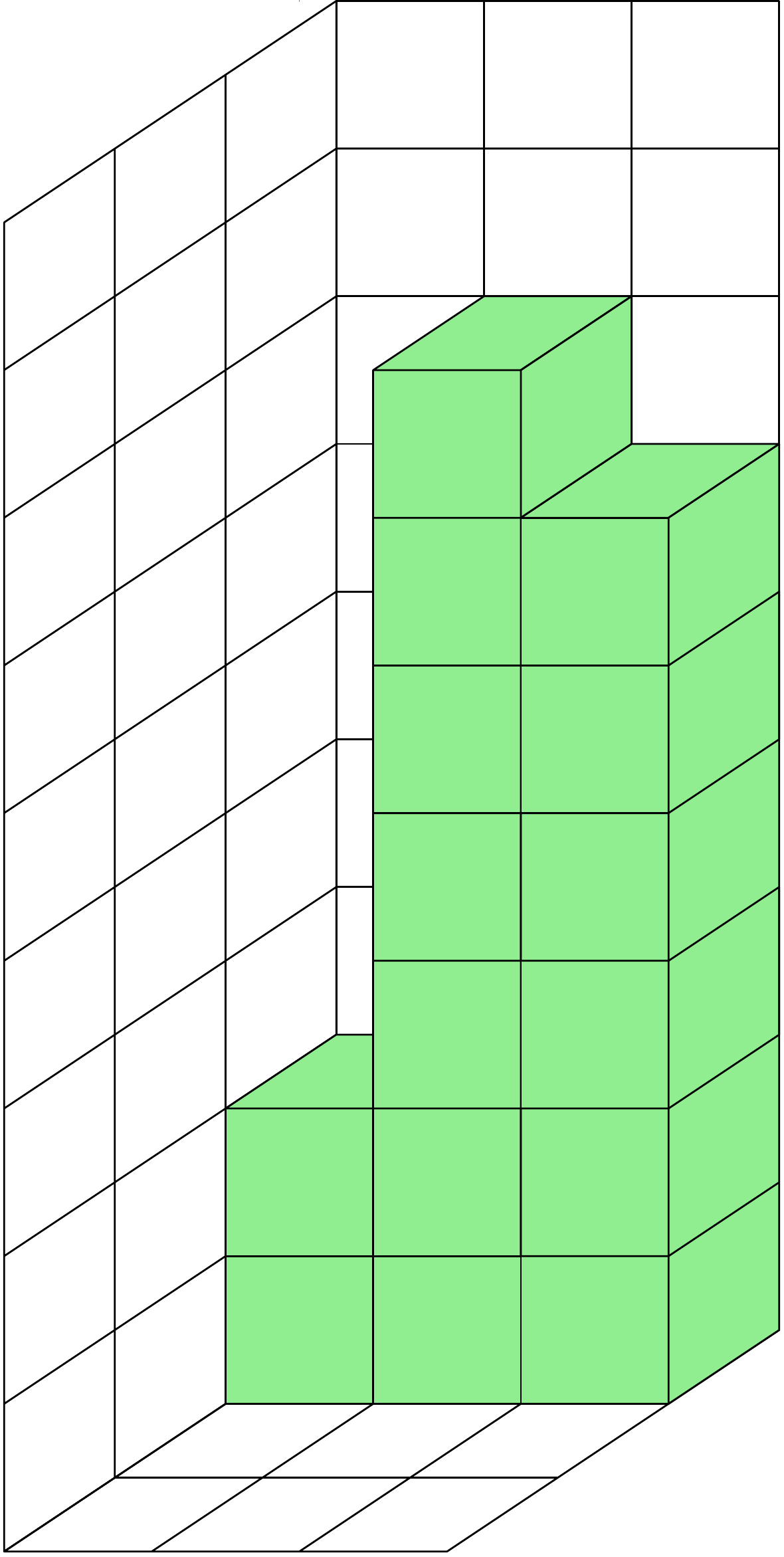}
	\end{minipage}\vspace{5pt}
	\caption{Representation of the numbers in the first row of a $3\times 3$ magic square  as columns of small cubes for the binary formulation\label{fig:bin}}
\end{figure}
The constraints from the previous subsection become (see \cref{vis}):
\begin{align*}
C_1 &:= \left\lbrace B \in \mathbb{R}^{n\times n \times n^2} : \sum_{j=1}^n  \sum_{k=1}^{n^2} B_{i,j,k}  = c \text{ for each } i \in I \right\rbrace,\\
C_2 &:= \left\lbrace B \in \mathbb{R}^{n\times n \times n^2} : \sum_{i=1}^n  \sum_{k=1}^{n^2} B_{i,j,k}  = c \text{ for each } j \in I \right\rbrace,\\
C_3 &:= \left\lbrace B \in \mathbb{R}^{n\times n \times n^2} : \sum_{i=1}^n \sum_{k=1}^{n^2} B_{i,i,k}  = c \right\rbrace,\\
C_4 &:= \left\lbrace B \in \mathbb{R}^{n\times n \times n^2} : \sum_{i=1}^n  \sum_{k=1}^{n^2} B_{i,n+1-i,k}= c \right\rbrace.
\end{align*}
\begin{figure}[ht!]
	\centering
	\subfloat[$C_1$]{
		\label{c_1}
		\includegraphics[width=0.15\textwidth]{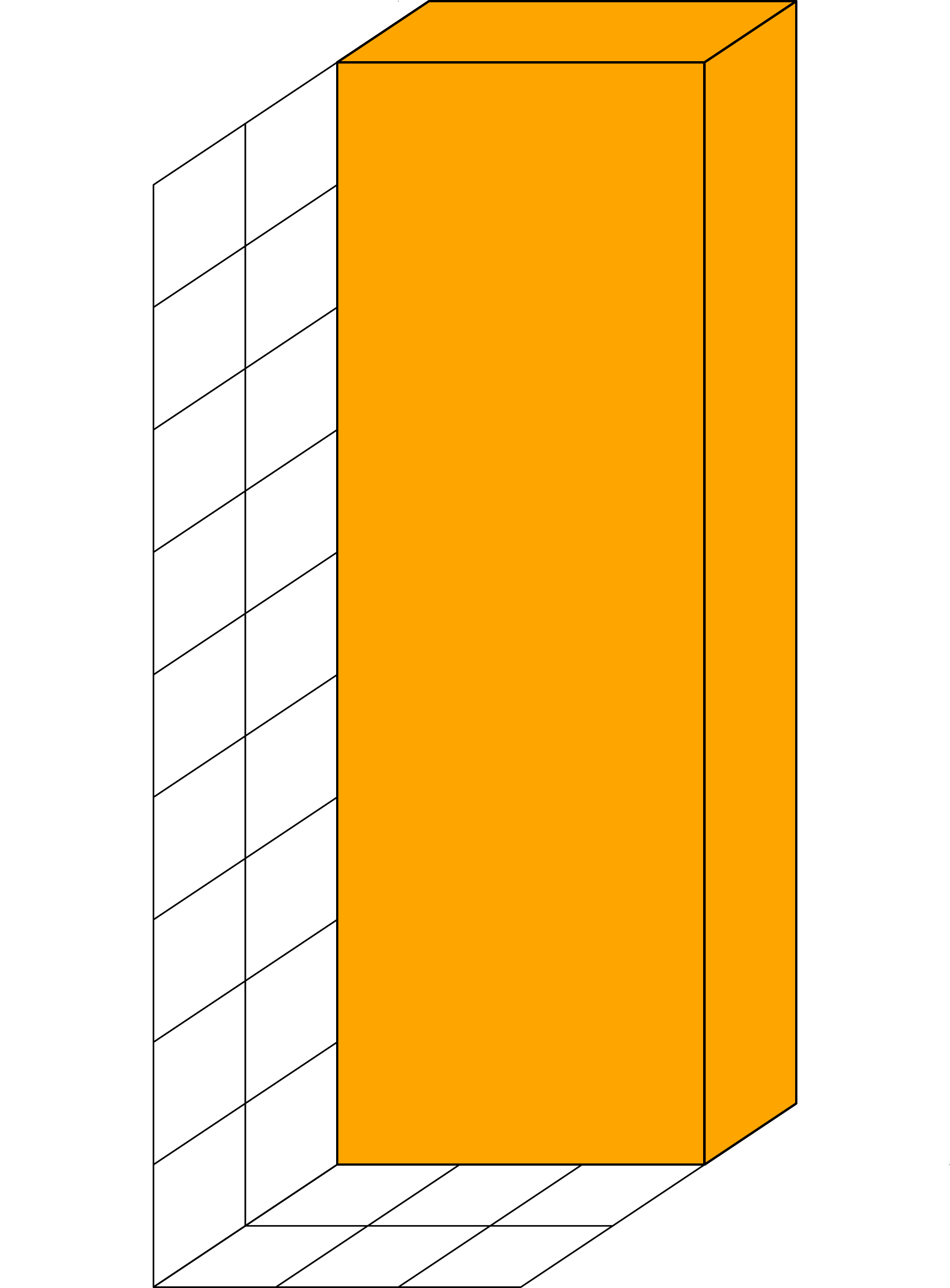}}
	\hspace{10pt}
	\subfloat[$C_2$]{
		\label{c_2}
		\includegraphics[width=0.15\textwidth]{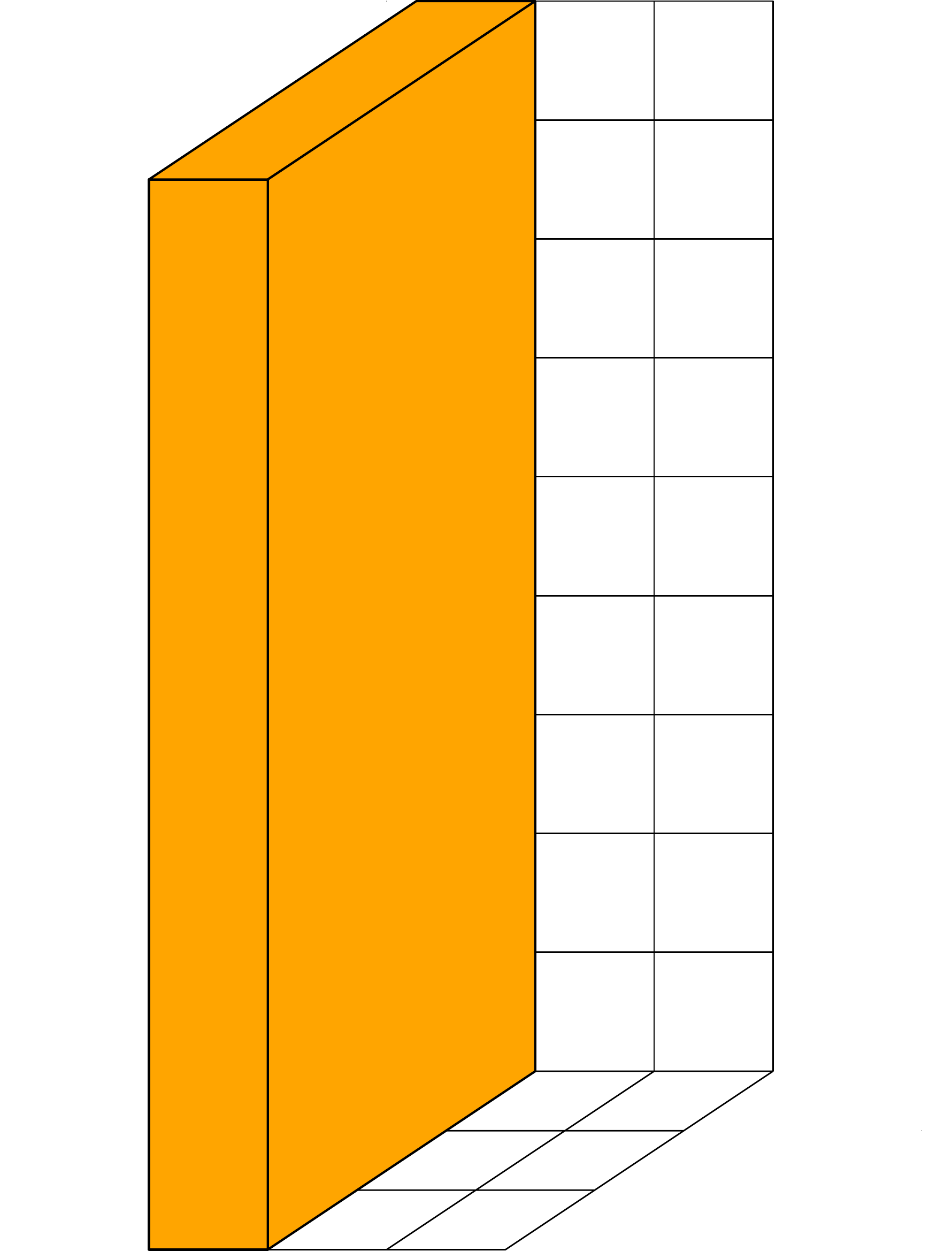}}
	\hspace{10pt}
	\subfloat[$C_3$]{
		\label{c_3}
		\includegraphics[width=0.15\textwidth]{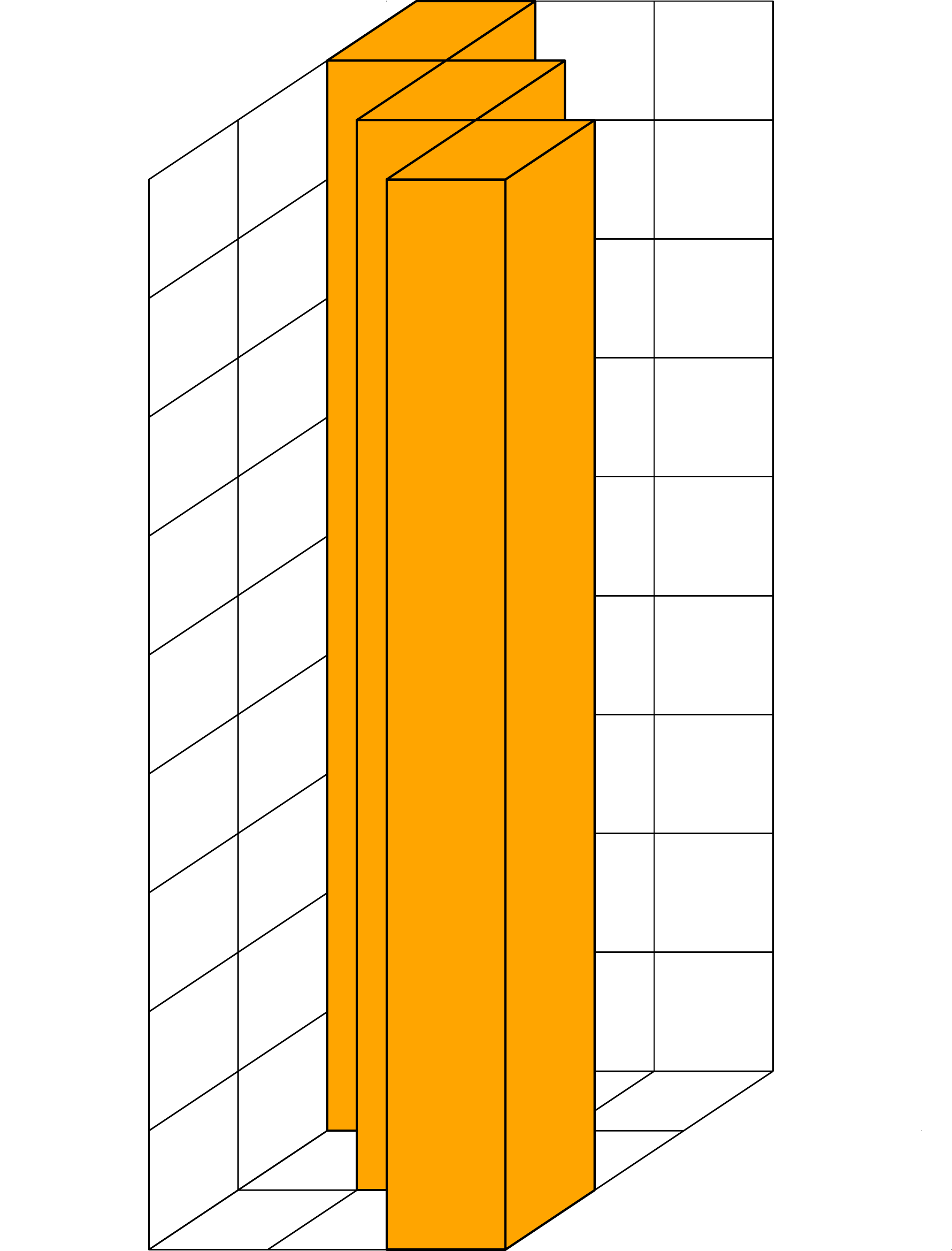}}
	\hspace{10pt}
	\subfloat[$C_4$]{
		\label{c_4}
		\includegraphics[width=0.15\textwidth]{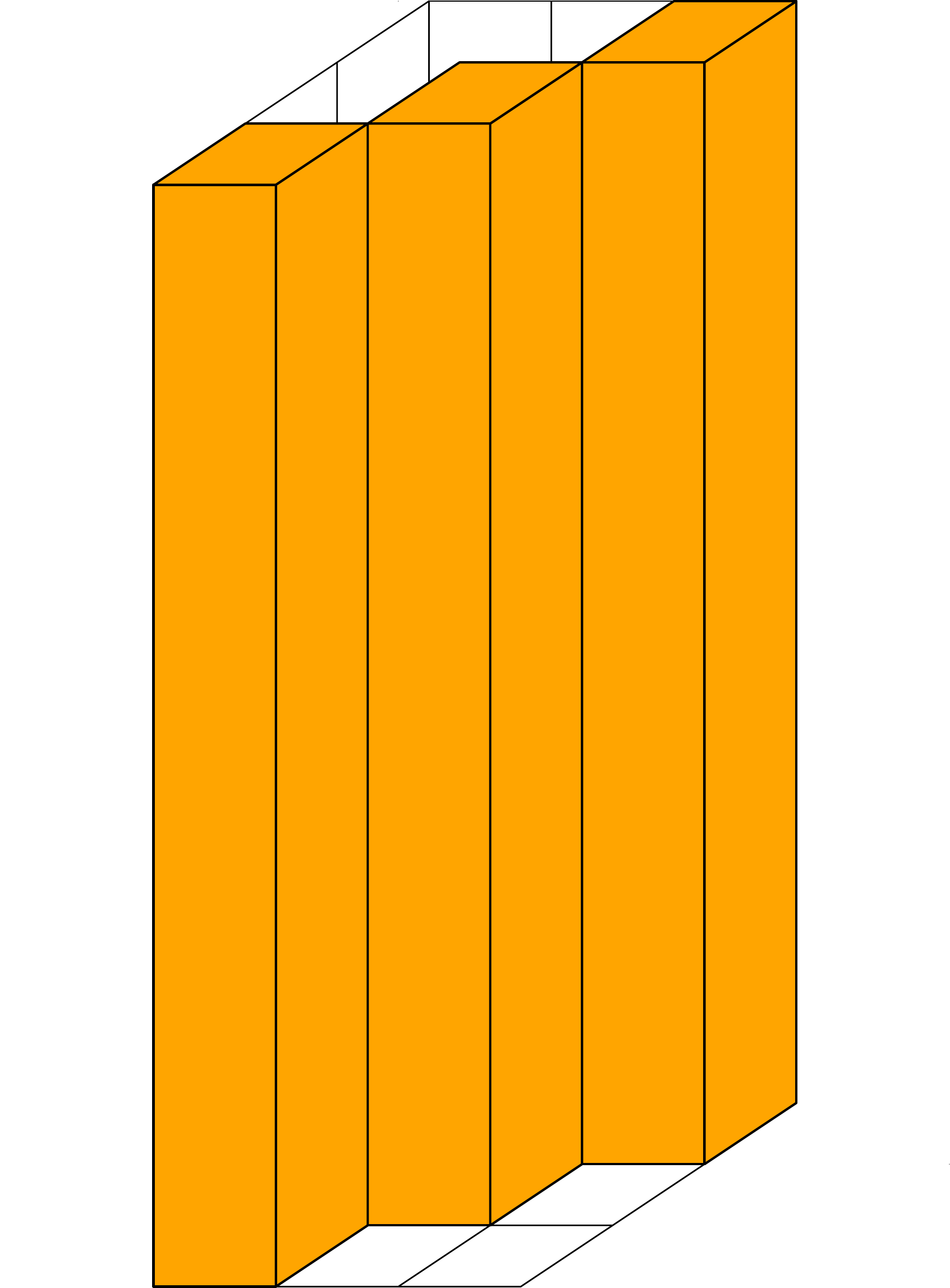}}
	\vspace{2pt}
	\label{rest}
	\caption{Visualization of the constraints used for modeling magic squares as a zero-one program. Each colored block must be formed by exactly $c$ cubes (ones)} \label{vis}
\end{figure}

Constraint $C_5$ of the integer formulation is now written for the binary formulation as the intersection of two sets:
\begin{align*}
C_5 &:= \left\lbrace B \in \{ 0,1 \}^{n \times n \times n^2} : \sum_{i=1}^n  \sum_{j=1}^{n} B_{i,j,k} = n^2 - k +1, %\text{ for each }
\forall k=1, \ldots, n^2 \right\rbrace,\\
C_6 &:= \left \lbrace B \in \{ 0,1 \}^{n \times n \times n^2} : \vect B_{i,j,:} \in \mathcal{B} \text{ for each } i,j \in I \right\rbrace,
\end{align*}
where $\mathcal{B}:=\left\lbrace [1,1,\ldots,1,1], [1,1,\ldots,1,0], \ldots, [1,0,\ldots,0,0] \right\rbrace$ is a base of vectors of $\mathbb{R}^{n^2}$.
On the one hand, constraint $C_5$ guarantees that the first \emph{floor} of $B$ is all filled with ones, the second must have $n^2-1$ ones, and so on until we get to the top \emph{floor}, which only contains a one. On the other hand, constraint $C_6$ guarantees that the matrix is formed by \emph{pillars} of ones \emph{standing} on the \emph{floor}, so if there is a one in an entry, all the elements below must be ones as well. Then, $B \in \mathbb{R}^{n\times n \times n^2}$ is a magic square if and only if
$B\in \bigcap_{i=1}^6 C_i.$

 The projections onto the first four sets are given by \cref{propi}:
\begin{align*}
P_{C_1}(B)&= B + \frac{1}{n^3}\left( \begin{array}{ccc} u_1 & \cdots & u_1\\
u_2 & \cdots & u_2\\
 \vdots & & \vdots \\
u_n & \cdots & u_n
\end{array} \right),\\
P_{C_2}(B)&= B + \frac{1}{n^3}\left( \begin{array}{cccc} v_1 & v_2 & \cdots & v_n\\
\vdots & \vdots & & \vdots \\
v_1 & v_2 & \cdots & v_n
\end{array} \right),\\
P_{C_3}(B)&= B + \frac{1}{n^3}\left(c-\sum_{i=1}^n\sum_{k=1}^{n^2} B_{iik}\right)\left( \begin{array}{cccc} e & 0_{n^2}& \cdots & 0_{n^2}\\
0_{n^2} & e & \ddots &\vdots\\
\vdots & \ddots  & \ddots & 0_{n^2}\\
0_{n^2} &  \cdots &0_{n^2} &e
\end{array} \right),\\
P_{C_4}(B)&= B + \frac{1}{n^3}\left(c-\sum_{i=1}^n\sum_{k=1}^{n^2} B_{i(n+i-1)k}\right) \left( \begin{array}{cccc} 0_{n^2} & \cdots & 0_{n^2} & e\\
\vdots &\udots & e &0_{n^2}\\
0_{n^2} & \udots  &  \udots& \vdots\\
e & 0_{n^2} & \cdots &0_{n^2}
\end{array} \right).
\end{align*}
where $e=(1,\ldots, 1) \in \mathbb{R}^{n^2}$ and
\begin{align*}
u_p&:=\left(c-\sum_{j=1}^{n}\sum_{k=1}^{n^2}  B_{p,j,k}, \ldots, c- \sum_{j=1}^{n}\sum_{k=1}^{n^2} B_{p,j,k}\right) \in \mathbb{R}^{n^2},\\
v_p&:=\left(c-\sum_{i=1}^{n}\sum_{k=1}^{n^2}  B_{i,p,k}, \ldots, c- \sum_{i=1}^{n}\sum_{k=1}^{n^2} B_{i,p,k}\right) \in \mathbb{R}^{n^2}.
\end{align*}

The projection onto $C_5$ assigns a one to all the elements of the first \emph{floor}, a zero to the smallest element and a one to the rest in the second \emph{floor}, a zero to the two smallest elements and a one to the rest in the third \emph{floor}, and so on until the top \emph{floor}, which only contains a one.
The projection onto $C_6$ can be computed componentwise as follows: For each $i,j\in I$, we have
$$P_{C_6}(B)_{i,j,:}=\argmin_{b\in\mathcal{B}}\|B_{i,j,:}-b\|.$$

% \begin{figure}[h]
% 	\centering
% 	\subfloat{
% 		\includegraphics[width=0.17\textwidth]{c_6}}
% 	\vspace{2pt}
% 	\label{rest}
% 	\caption{Visualization of a block $B$ of cubes} \label{block}
% \end{figure}

%\begin{center}
%\includegraphics[scale=0.18]{c_6.pdf}
%\end{center}

\begin{remark}
 The formulation of the problem of completing a partially filled magic square  is analogous to the integer formulation, so we omit the details.
\end{remark}

%\subsubsection*{Completing a partially filled magic square}
%
%If $M$ is an $n \times n \times n^2$ partially complete array that represents an incomplete magic square, denote by $J' \subseteq I^3$ the set of indices for which we have a value in the magic square. Then it's necessary to add the constraint
%$$C_7 = \left\lbrace B : B_{i,j,k}=1  \text{ for each } (i,j,k) \in J \right\rbrace $$
%to the set of constraints we already have.
%The projection onto $C_7$ is given componentwise by
%$$
%P_{C_7} (B_{i,j,k}) = \begin{cases}
%M_{i,j,k} & \text{if } (i,j,k) \in J',\\
%B_{i,j,k} & \text{otherwise,}
%\end{cases}
%$$
%for each $(i,j,k) \in I^3$.  Then, $B$ completes $M$ if and only if $B \in \bigcap_{i=1}^7 C_i.$

\section{Numerical experiments}\label{sec:num_exp}

In this section we run a numerical experiment to test the performance of the Douglas--Rachford
algorithm for finding magic squares of order $n=3,4,\ldots,12$, on which we compare the integer and binary formulations discussed in \cref{sec:3}. All our codes were written in Python~2.7 and the tests were run on an Intel Core i7-4770 CPU \@3.40GHz with 32GB RAM, under Windows 10 (64-bit).

For each order and each formulation, the algorithm was run on the corresponding product space (see \cref{sec:product_space}) for $100$ random starting points.
For the integer formulation, we took a starting point of the form $x_0:=(y,y,y,y,y) \in D$ for some $y \in \mathbb{R}^{n \times n}$ randomly chosen with entries in $]0,1[$, while for the binary formulation, we choose $x_0:=(y,y,y,y,y,y) \in D$ for some random $y \in \mathbb{R}^{n\times n \times n^2}$ with entries in $]0,1[$. The iterates were then defined by
$$x_{n+1}=T_{D,C}(x_n)\quad\text{for }n=0,1,\ldots.$$
The Douglas--Rachford algorithm was stopped either after a maximum time of $1800$ seconds ($30$ minutes) or when
$$\| \operatorname{round}(P_D(x_n)) - P_{C_i}(\operatorname{round}(P_D(x_n)))\|\leq 0.05\quad\text{ for all }i,$$
where $\operatorname{round}(\cdot)$ gives the nearest integer componentwise.

\cref{table:porc} summarizes the success rate and running time for each formulation. We observe how the binary formulation was not able to solve any instance for orders larger than $6$. On the other hand, for orders $4$ and $5$, it was more successful than the integer formulation, which surprisingly becomes more successful for larger orders than for smaller ones. The algorithm was much faster with the integer formulation.
%
%\begin{table}[htbp]
%\begin{center}
%\begin{tabular}{ccccccccccc}
% & 3 & 4 & 5 & 6 & 7 & 8 & 9  & 10 & 11 & 12\\
%\hline
%Binary  &  99 & 93 & 79 &  3 &  0 &  0 &  0 &  0 &  0 &  0\\ \hline
%Integer & 100 & 64 & 59 & 80 & 86 & 94 & 96 & 94 & 97 & 99\\ \hline
%\end{tabular}
%\caption{$\%$ of magic squares found in less than $30$ minutes}
%\label{table:porc}
%\end{center}
%\end{table}
%
%\begin{table}[htbp]
%\begin{center}
%\begin{tabular}{ccccccccccc}
% & 3 & 4 & 5 & 6 & 7 & 8 & 9  & 10 & 11 & 12\\
%\hline
%Binary  &  0.05   &  1.67 &  436.86 & 1228.78 &  - &  - &  - &  - &  - &  -\\ \hline
%Integer & 0.01 & 0.06 & 0.67 & 0.53 & 0.74 & 0.86 & 0.9 &  1.47 & 2.45 & 4.86\\ \hline
%\end{tabular}
%\caption{Average running time of successful magic squares found}
%\label{table:porc}
%\end{center}
%\end{table}

\begin{table}[htbp]
\begin{center}
\begin{tabular}{|c|cd{3.3}d{3.3}||cd{3.3}d{3.3}|}
\cline{2-7}
\multicolumn{1}{c|}{}& \multicolumn{3}{c||}{Binary} & \multicolumn{3}{c|}{Integer}\\
\hline
 Order & Success & \multicolumn{2}{c||}{\text{Time}}  & Success & \multicolumn{2}{c|}{\text{Time}}\\
 \hline
 3 & 99 &   0.05 & (0.17) & 100 & 0.01 &(0.02)\\
 4 & 93 &   1.67 & (13.02) &  64 & 0.06 &(2.64)\\
 5 & 79 & 436.86 & (1739.93) &  59 & 0.67 &(5.63)\\
 6 &  3 &1228.78 & (1742.74) &  80 & 0.53 &(2.55)\\
 7 &  0 & \multicolumn{2}{c||}{-} &  86 & 0.74& (4.92)\\
 8 &  0 & \multicolumn{2}{c||}{-}      &  94 & 0.86& (4.39)\\
 9 &  0 & \multicolumn{2}{c||}{-}      &  96 & 0.90 &(3.39)\\
10 &  0 & \multicolumn{2}{c||}{-}      &  94 & 1.47 &(8.69)\\
11 &  0 & \multicolumn{2}{c||}{-}      &  97 & 2.45 &(13.97)\\
12 &  0 & \multicolumn{2}{c||}{-}      &  99 & 4.86 &(22.97)\\
13 &  0 & \multicolumn{2}{c||}{-}      &  98 & 9.64 &(56.71)\\
14 &  0 & \multicolumn{2}{c||}{-}      &  100 & 24.31 &(84.90)\\
15 &  0 & \multicolumn{2}{c||}{-}      &  100 & 60.04 &(236.27)\\
16 &  0 & \multicolumn{2}{c||}{-}      &  100 & 245.58 &(1291.61)\\
\hline
\end{tabular}\vspace{5pt}
\caption{Magic squares found in less than $30$ minutes and mean (max) time}
\label{table:porc}
\end{center}
\end{table}

To better understand the behavior of the Douglas--Rachford algorithm for both formulations, we show in \cref{fig:cumfreq} the cumulative frequency over time. Note that we have used a logarithmic scale in the horizontal axis, to show how the running time of the algorithm exponentially grows with the order $n$. Apparently, the reason why the algorithm failed in the binary formulation for $n$ larger than $6$ is that it did not have enough time to converge. We show in \cref{fig:magicsquares} a magic square of orders between $3$ and $12$ found with the integer formulation.

\begin{figure}[ht!]
\centering
\includegraphics[width=\textwidth]{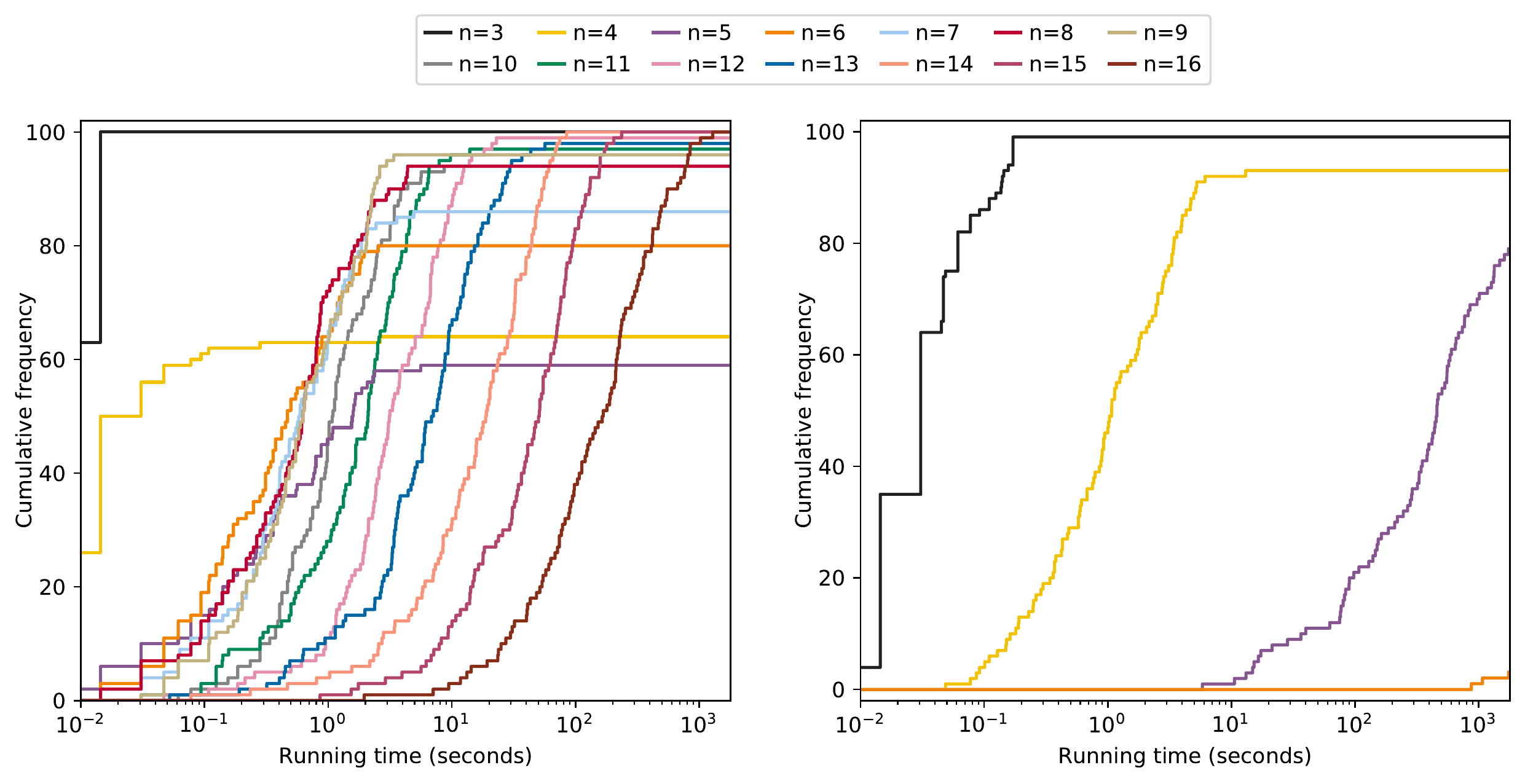}
%\subfloat[Binary formulation]{\includegraphics[width=0\textwidth]{CumFreq_binary.pdf}}
%\subfloat[Integer formulation]{\includegraphics[width=0\textwidth]{CumFreq_Integer.pdf}}\vspace{5pt}
\caption{Cumulative frequencies over running time of the integer (left) and binary (right) formulations}\label{fig:cumfreq}
\end{figure}

\renewcommand{\arraystretch}{1.3}
\begin{sidewaysfigure}[htbp]
\scriptsize\begin{center}
\begin{tabular}{|C|C|C|C|C|C|C|C|C|C|C|CCCCCC|C|C|C|C|C|C|C|C|C|C|C|C|}
\cline{1-3} \cline{5-8} \cline{10-14} \cline{16-21} \cline{23-29}
8 & 3 & 4 &  & 11 & 14 & 5 & 4 &  & 2 & 14 & \multicolumn{1}{C|}{20} & \multicolumn{1}{C|}{11} & \multicolumn{1}{C|}{18} & \multicolumn{1}{C|}{} & \multicolumn{1}{C|}{6} & 35 & 2 & 31 & 29 & 8 &  & 13 & 4 & 33 & 28 & 38 & 18 & 41\tabularnewline
\cline{1-3} \cline{5-8} \cline{10-14} \cline{16-21} \cline{23-29}
1 & 5 & 9 &  & 16 & 2 & 9 & 7 &  & 19 & 9 & \multicolumn{1}{C|}{5} & \multicolumn{1}{C|}{8} & \multicolumn{1}{C|}{24} & \multicolumn{1}{C|}{} & \multicolumn{1}{C|}{5} & 27 & 17 & 9 & 19 & 34 &  & 10 & 22 & 30 & 34 & 25 & 17 & 37\tabularnewline
\cline{1-3} \cline{5-8} \cline{10-14} \cline{16-21} \cline{23-29}
6 & 7 & 2 &  & 1 & 15 & 8 & 10 &  & 15 & 17 & \multicolumn{1}{C|}{23} & \multicolumn{1}{C|}{3} & \multicolumn{1}{C|}{7} & \multicolumn{1}{C|}{} & \multicolumn{1}{C|}{14} & 13 & 23 & 15 & 16 & 30 &  & 43 & 44 & 15 & 29 & 14 & 6 & 24\tabularnewline
\cline{1-3} \cline{5-8} \cline{10-14} \cline{16-21} \cline{23-29}
\multicolumn{1}{C}{} & \multicolumn{1}{C}{} & \multicolumn{1}{C}{} &  & 6 & 3 & 12 & 13 &  & 25 & 12 & \multicolumn{1}{C|}{1} & \multicolumn{1}{C|}{21} & \multicolumn{1}{C|}{6} & \multicolumn{1}{C|}{} & \multicolumn{1}{C|}{32} & 10 & 26 & 12 & 28 & 3 &  & 16 & 47 & 1 & 31 & 5 & 48 & 27\tabularnewline
\cline{5-8} \cline{10-14} \cline{16-21} \cline{23-29}
\multicolumn{1}{C}{} & \multicolumn{1}{C}{} & \multicolumn{1}{C}{} & \multicolumn{1}{C}{} & \multicolumn{1}{C}{} & \multicolumn{1}{C}{} & \multicolumn{1}{C}{} & \multicolumn{1}{C}{} &  & 4 & 13 & \multicolumn{1}{C|}{16} & \multicolumn{1}{C|}{22} & \multicolumn{1}{C|}{10} & \multicolumn{1}{C|}{} & \multicolumn{1}{C|}{33} & 22 & 7 & 20 & 18 & 11 &  & 8 & 26 & 21 & 23 & 49 & 45 & 3\tabularnewline
\cline{10-14} \cline{16-21} \cline{23-29}
\multicolumn{1}{C}{} & \multicolumn{1}{C}{} & \multicolumn{1}{C}{} & \multicolumn{1}{C}{} & \multicolumn{1}{C}{} & \multicolumn{1}{C}{} & \multicolumn{1}{C}{} & \multicolumn{1}{C}{} & \multicolumn{1}{C}{} & \multicolumn{1}{C}{} & \multicolumn{1}{C}{} &  &  &  & \multicolumn{1}{C|}{} & \multicolumn{1}{C|}{21} & 4 & 36 & 24 & 1 & 25 &  & 46 & 12 & 40 & 19 & 42 & 9 & 7\tabularnewline
\cline{16-21} \cline{23-29}
\multicolumn{1}{C}{} & \multicolumn{1}{C}{} & \multicolumn{1}{C}{} & \multicolumn{1}{C}{} & \multicolumn{1}{C}{} & \multicolumn{1}{C}{} & \multicolumn{1}{C}{} & \multicolumn{1}{C}{} & \multicolumn{1}{C}{} & \multicolumn{1}{C}{} & \multicolumn{1}{C}{} &  &  &  &  &  & \multicolumn{1}{C}{} & \multicolumn{1}{C}{} & \multicolumn{1}{C}{} & \multicolumn{1}{C}{} & \multicolumn{1}{C}{} &  & 39 & 20 & 35 & 11 & 2 & 32 & 36\tabularnewline
\cline{23-29}
\multicolumn{1}{C}{} & \multicolumn{1}{C}{} & \multicolumn{1}{C}{} & \multicolumn{1}{C}{} & \multicolumn{1}{C}{} & \multicolumn{1}{C}{} & \multicolumn{1}{C}{} & \multicolumn{1}{C}{} & \multicolumn{1}{C}{} & \multicolumn{1}{C}{} & \multicolumn{1}{C}{} &  &  &  &  &  & \multicolumn{1}{C}{} & \multicolumn{1}{C}{} & \multicolumn{1}{C}{} & \multicolumn{1}{C}{} & \multicolumn{1}{C}{} & \multicolumn{1}{C}{} & \multicolumn{1}{C}{} & \multicolumn{1}{C}{} & \multicolumn{1}{C}{} & \multicolumn{1}{C}{} & \multicolumn{1}{C}{} & \multicolumn{1}{C}{} & \multicolumn{1}{C}{}\tabularnewline
\cline{1-8} \cline{10-18} \cline{20-29}
10 & 15 & 41 & 58 & 54 & 23 & 34 & 25 &  & 71 & 14 & \multicolumn{1}{C|}{40} & \multicolumn{1}{C|}{68} & \multicolumn{1}{C|}{28} & \multicolumn{1}{C|}{19} & \multicolumn{1}{C|}{22} & 70 & 37 &  & 76 & 64 & 57 & 21 & 17 & 59 & 97 & 33 & 74 & 7\tabularnewline
\cline{1-8} \cline{10-18} \cline{20-29}
33 & 55 & 3 & 28 & 49 & 47 & 38 & 7 &  & 46 & 29 & \multicolumn{1}{C|}{59} & \multicolumn{1}{C|}{21} & \multicolumn{1}{C|}{58} & \multicolumn{1}{C|}{62} & \multicolumn{1}{C|}{51} & 10 & 33 &  & 69 & 49 & 34 & 39 & 75 & 9 & 40 & 52 & 48 & 90\tabularnewline
\cline{1-8} \cline{10-18} \cline{20-29}
14 & 50 & 32 & 1 & 59 & 57 & 42 & 5 &  & 43 & 64 & \multicolumn{1}{C|}{1} & \multicolumn{1}{C|}{66} & \multicolumn{1}{C|}{47} & \multicolumn{1}{C|}{24} & \multicolumn{1}{C|}{78} & 16 & 30 &  & 66 & 24 & 22 & 72 & 100 & 78 & 81 & 6 & 2 & 54\tabularnewline
\cline{1-8} \cline{10-18} \cline{20-29}
51 & 40 & 27 & 18 & 11 & 43 & 31 & 39 &  & 4 & 73 & \multicolumn{1}{C|}{36} & \multicolumn{1}{C|}{5} & \multicolumn{1}{C|}{34} & \multicolumn{1}{C|}{77} & \multicolumn{1}{C|}{52} & 50 & 38 &  & 19 & 37 & 68 & 84 & 14 & 18 & 92 & 55 & 91 & 27\tabularnewline
\cline{1-8} \cline{10-18} \cline{20-29}
20 & 30 & 61 & 63 & 17 & 4 & 36 & 29 &  & 3 & 20 & \multicolumn{1}{C|}{67} & \multicolumn{1}{C|}{56} & \multicolumn{1}{C|}{65} & \multicolumn{1}{C|}{26} & \multicolumn{1}{C|}{17} & 55 & 60 &  & 8 & 73 & 10 & 13 & 12 & 89 & 96 & 70 & 51 & 83\tabularnewline
\cline{1-8} \cline{10-18} \cline{20-29}
44 & 37 & 22 & 13 & 12 & 62 & 6 & 64 &  & 45 & 49 & \multicolumn{1}{C|}{69} & \multicolumn{1}{C|}{11} & \multicolumn{1}{C|}{12} & \multicolumn{1}{C|}{75} & \multicolumn{1}{C|}{44} & 25 & 39 &  & 58 & 31 & 62 & 86 & 87 & 23 & 32 & 95 & 5 & 26\tabularnewline
\cline{1-8} \cline{10-18} \cline{20-29}
53 & 9 & 48 & 19 & 56 & 8 & 21 & 46 &  & 74 & 80 & \multicolumn{1}{C|}{2} & \multicolumn{1}{C|}{57} & \multicolumn{1}{C|}{35} & \multicolumn{1}{C|}{31} & \multicolumn{1}{C|}{9} & 18 & 63 &  & 61 & 79 & 38 & 60 & 35 & 93 & 16 & 11 & 71 & 41\tabularnewline
\cline{1-8} \cline{10-18} \cline{20-29}
35 & 24 & 26 & 60 & 2 & 16 & 52 & 45 &  & 7 & 13 & \multicolumn{1}{C|}{54} & \multicolumn{1}{C|}{79} & \multicolumn{1}{C|}{42} & \multicolumn{1}{C|}{32} & \multicolumn{1}{C|}{81} & 53 & 8 &  & 67 & 46 & 43 & 50 & 25 & 44 & 1 & 99 & 45 & 85\tabularnewline
\cline{1-8} \cline{10-18} \cline{20-29}
\multicolumn{1}{C}{} & \multicolumn{1}{C}{} & \multicolumn{1}{C}{} & \multicolumn{1}{C}{} & \multicolumn{1}{C}{} & \multicolumn{1}{C}{} & \multicolumn{1}{C}{} & \multicolumn{1}{C}{} &  & 76 & 27 & \multicolumn{1}{C|}{41} & \multicolumn{1}{C|}{6} & \multicolumn{1}{C|}{48} & \multicolumn{1}{C|}{23} & \multicolumn{1}{C|}{15} & 72 & 61 &  & 28 & 20 & 94 & 65 & 42 & 29 & 3 & 80 & 88 & 56\tabularnewline
\cline{10-18} \cline{20-29}
\multicolumn{1}{C}{} & \multicolumn{1}{C}{} & \multicolumn{1}{C}{} & \multicolumn{1}{C}{} & \multicolumn{1}{C}{} & \multicolumn{1}{C}{} & \multicolumn{1}{C}{} & \multicolumn{1}{C}{} & \multicolumn{1}{C}{} & \multicolumn{1}{C}{} & \multicolumn{1}{C}{} &  &  &  &  &  & \multicolumn{1}{C}{} & \multicolumn{1}{C}{} &  & 53 & 82 & 77 & 15 & 98 & 63 & 47 & 4 & 30 & 36\tabularnewline
\cline{20-29}
\multicolumn{1}{C}{} & \multicolumn{1}{C}{} & \multicolumn{1}{C}{} & \multicolumn{1}{C}{} & \multicolumn{1}{C}{} & \multicolumn{1}{C}{} & \multicolumn{1}{C}{} & \multicolumn{1}{C}{} & \multicolumn{1}{C}{} & \multicolumn{1}{C}{} & \multicolumn{1}{C}{} &  &  &  &  &  & \multicolumn{1}{C}{} & \multicolumn{1}{C}{} & \multicolumn{1}{C}{} & \multicolumn{1}{C}{} & \multicolumn{1}{C}{} & \multicolumn{1}{C}{} & \multicolumn{1}{C}{} & \multicolumn{1}{C}{} & \multicolumn{1}{C}{} & \multicolumn{1}{C}{} & \multicolumn{1}{C}{} & \multicolumn{1}{C}{} & \multicolumn{1}{C}{}\tabularnewline
\cline{1-11} \cline{18-29}
114 & 61 & 60 & 36 & 88 & 2 & 78 & 73 & 52 & 72 & 35 &  &  &  &  &  &  & 19 & 131 & 79 & 107 & 103 & 58 & 59 & 35 & 16 & 109 & 60 & 94\tabularnewline
\cline{1-11} \cline{18-29}
119 & 46 & 81 & 117 & 98 & 6 & 42 & 26 & 20 & 91 & 25 &  &  &  &  &  &  & 104 & 25 & 61 & 38 & 87 & 138 & 37 & 140 & 124 & 39 & 31 & 46\tabularnewline
\cline{1-11} \cline{18-29}
93 & 7 & 40 & 104 & 109 & 17 & 77 & 31 & 79 & 69 & 45 &  &  &  &  &  &  & 67 & 30 & 64 & 137 & 49 & 69 & 17 & 101 & 72 & 34 & 113 & 117\tabularnewline
\cline{1-11} \cline{18-29}
3 & 71 & 51 & 32 & 83 & 80 & 84 & 15 & 59 & 75 & 118 &  &  &  &  &  &  & 76 & 126 & 55 & 132 & 13 & 23 & 84 & 73 & 57 & 82 & 129 & 20\tabularnewline
\cline{1-11} \cline{18-29}
62 & 105 & 37 & 21 & 54 & 113 & 102 & 48 & 9 & 10 & 110 &  &  &  &  &  &  & 50 & 108 & 134 & 68 & 78 & 71 & 9 & 97 & 41 & 24 & 88 & 102\tabularnewline
\cline{1-11} \cline{18-29}
28 & 11 & 64 & 38 & 47 & 87 & 65 & 97 & 89 & 29 & 116 &  &  &  &  &  &  & 51 & 112 & 22 & 7 & 114 & 139 & 128 & 10 & 75 & 111 & 47 & 54\tabularnewline
\cline{1-11} \cline{18-29}
44 & 16 & 115 & 120 & 24 & 94 & 18 & 106 & 14 & 50 & 70 &  &  &  &  &  &  & 8 & 123 & 4 & 89 & 48 & 127 & 26 & 70 & 121 & 40 & 98 & 116\tabularnewline
\cline{1-11} \cline{18-29}
99 & 112 & 23 & 12 & 5 & 101 & 86 & 43 & 49 & 74 & 67 &  &  &  &  &  &  & 86 & 5 & 28 & 2 & 91 & 100 & 125 & 136 & 66 & 106 & 80 & 45\tabularnewline
\cline{1-11} \cline{18-29}
34 & 68 & 96 & 33 & 13 & 30 & 92 & 55 & 103 & 90 & 57 &  &  &  &  &  &  & 133 & 12 & 141 & 11 & 93 & 3 & 130 & 144 & 18 & 56 & 77 & 52\tabularnewline
\cline{1-11} \cline{18-29}
53 & 108 & 41 & 100 & 39 & 56 & 8 & 82 & 76 & 107 & 1 &  &  &  &  &  &  & 143 & 43 & 120 & 83 & 99 & 33 & 44 & 21 & 110 & 85 & 36 & 53\tabularnewline
\cline{1-11} \cline{18-29}
22 & 66 & 63 & 58 & 111 & 85 & 19 & 95 & 121 & 4 & 27 &  &  &  &  &  &  & 118 & 65 & 27 & 115 & 63 & 95 & 92 & 42 & 96 & 122 & 6 & 29\tabularnewline
\cline{1-11} \cline{18-29}
\multicolumn{1}{C}{} & \multicolumn{1}{C}{} & \multicolumn{1}{C}{} & \multicolumn{1}{C}{} & \multicolumn{1}{C}{} & \multicolumn{1}{C}{} & \multicolumn{1}{C}{} & \multicolumn{1}{C}{} & \multicolumn{1}{C}{} & \multicolumn{1}{C}{} & \multicolumn{1}{C}{} &  &  &  &  &  &  & 15 & 90 & 135 & 81 & 32 & 14 & 119 & 1 & 74 & 62 & 105 & 142\tabularnewline
\cline{18-29}
\end{tabular}\vspace{3pt}
\end{center}
\caption{Magic squares of orders 3 to 12 found with the integer formulation}\label{fig:magicsquares}
\end{sidewaysfigure}

\section{Conclusions}\label{sec:5}

In this article, we have shown that the Douglas--Rachford algorithm can be used as a successful heuristic for constructing magic squares, despite the nonconvexity of the problem. To this aim, we have proposed two different formulations as a feasibility problem, one integer and another one binary. Our numerical tests demonstrate that the integer formulation is more effective for finding magic squares because it requires considerably much less time on average for finding a solution. Surprisingly, this behavior was the opposite with the similar binary and integer formulations for solving Sudoku puzzles, where the integer formulation is completely unsuccessful. Interestingly, the binary was more successful than the integer formulation for orders $4$ and $5$. For higher orders, the binary formulation was not able to solve any instance, possibly due to lack of time.

\section*{Acknowledgements}
The first author was supported by MINECO of Spain and ERDF of EU, as part of the Ram\'on y Cajal
program (RYC-2013-13327) and the Grant MTM2014-59179-C2-1-P.

%\section{Acerca de las referencias}
%
%Las referencias en el texto se deben citar por los apellidos de los
%autores seguidos del año de publicación entre paréntesis. Algunos
%ejemplos:
%
%En Apellido Autor 1 (2000) se estudia ...
%
%El concepto está definido en Apellido Autor 2 y Apellido Autor 3
%(2001).
%
%Este tema ha sido estudiado ampliamente en Apellido Autor 1 (2000),
%Apellido Autor 2 y Apellido Autor 3 (2001), Apellido Autor 4 et al.
%(2002).
%
%La lista de referencias sólo debe incluir trabajos que están citados
%en el texto y que han sido publicados o aceptados para su
%publicación. Las comunicaciones personales y los trabajos no
%publicados sólo se deben mencionar en el texto. No utilizar notas al
%pie para sustituir una lista de referencias.
%
%La lista de referencias debe ser ordenada alfabéticamente por los
%apellidos del primer autor de cada trabajo.
%
%Utilizar siempre la abreviatura del nombre de la revista de acuerdo
%al ISSN List of Title Word Abbreviations, véase
%www.issn.org/2-22661-LTWA-online.php

\subsection*{About the authors}%

\textbf{Francisco J. Arag\'on Artacho} earned his Ph.D. in 2007 from the University of Murcia (Spain). He subsequently held a postdoctoral position ``Juan de la Cierva'' at the University of Alicante. In 2011 he moved to the University of Newcastle (Australia), where he became a Research Associate at the Research Centre for Computer-Assisted Research Mathematics and its Applications. Two years later, he moved to Luxembourg where he became a Marie Curie Fellow (AFR Postdoc) at the Luxembourg Centre for Systems Biomedicine. In November 2014 he started his current position at the University of Alicante as a Ram\'on y Cajal fellow. His main research interests include set-valued and variational analysis, optimization, convex analysis and projection algorithms.
\vspace{0.3cm}

\setlength{\parindent}{0pt}\textbf{Paula Segura Mart\'inez} graduated from the University of Alicante with a B.S. in Mathematics in 2017. She received her M.S. in Mathematics from the University of Murcia in 2018 and now she is a Ph.D. student at the University of Valencia. Her research interests are in the fields of combinatorial optimization, location and arc routing problems.

\end{document}